\documentclass[12pt,a4paper]{amsart}
\usepackage{graphicx}
\usepackage{pgf,tikz,pgfplots}
\usetikzlibrary{arrows}
\usepackage{comment}

\makeatletter
\long\def\unmarkedfootnote#1{{\long\def\@makefntext##1{##1}\footnotetext{#1}}}
\makeatother

\usepackage{amsmath,amssymb}
\usepackage{enumerate,amssymb, hyperref}
\usepackage{color}
\theoremstyle{plain}
\newtheorem{thm}{Theorem}[section]

\newtheorem{lemma}[thm]{Lemma}

\newtheorem{corollary}[thm]{Corollary}
\newtheorem{prop}[thm]{Proposition}
\newtoks\prt
\newtheorem{proclaim}[thm]{\the\prt}
\theoremstyle{definition}

\newtheorem{remark}[thm]{Remark}

\newtheorem{definition}[thm]{Definition}

\def\eqn#1$$#2$${\begin{equation}\label#1#2\end{equation}}

\numberwithin{equation}{section}

\def\adj{\operatorname{adj}}

\newcommand{\cL}{{\mathcal L}}
\newcommand{\cC}{{\mathcal C}}

\def\dist{\operatorname{dist}}

\def\epsilon{\varepsilon}
\def\en{\mathbb N}
\def\er{\mathbb R}

\def\f{\tilde{f}}

\def\haus{\mathcal{H}}

\def\id{\operatorname{id}}

\def\loc{\operatorname{loc}}

\def\mir1{\mathcal L_1}

\newcommand{\mV}{{\mathbb V}}
\newcommand{\mR}{{\mathbb R}}
\newcommand{\mv}{{\mathbf v}}

\def\rn{\mathbb R^n}

\def\x{\widetilde{x}}
\def\y{\widetilde{y}}
\def\ve{\boldsymbol v}

\makeatletter
\newcommand{\labeltext}[2]{%
	\@bsphack
	\def\@currentlabel{#1}{\label{#2}}%
	\@esphack
}
\makeatother

\def\step#1#2#3{\par \noindent{{\vskip 5pt \textit{Step}~\labeltext{#1}{#3}#1. }{\textit{#2.}}}}

\newtoks\by
\newtoks\paper
\newtoks\book
\newtoks\jour
\newtoks\yr
\newtoks\pages
\newtoks\vol
\newtoks\publ
\def\ota{{\hbox\vol{???}}}
\def\cLear{\by=\ota\paper=\ota\book=\ota\jour=\ota\yr=\ota
\pages=\ota\vol=\ota\publ=\ota}
\def\endpaper{\the\by, {\the\paper},
\textit{\the\jour} \textbf{\the\vol} (\the\yr), \the\pages.\cLear}
\def\endbook{\the\by, \textit{\the\book}, \the\publ.\cLear}
\def\endprep{\the\by, \textit{\the\paper}, \the\jour.\cLear}
\def\endyearprep{\the\by, \textit{\the\paper}, \the\jour, (\the\yr).\cLear}
\def\name#1#2{#2 #1}
\def\nom{ \rm no. }

\definecolor{ffqqqq}{rgb}{1.,0.,0.}
\definecolor{qqffqq}{rgb}{0.,1.,0.}
\definecolor{qqqqff}{rgb}{0.,0.,1.}
\definecolor{xfqqff}{rgb}{0.4980392156862745,0.,1.}
\definecolor{ffxfqq}{rgb}{1.,0.4980392156862745,0.}

\title{Mission $p<n-1$: Possible\\
 -- Nonlinear Elasticity Beyond Conventional Limits}

\author[D. Campbell]{Daniel Campbell}
\address{Department of Mathematical Analysis, Charles University, So\-ko\-lovsk\'a 83, 186~00 Prague 8, Czech Republic}
\email{\tt campbell@karlin.mff.cuni.cz}

\author[A. Dole\v{z}alov\'a]{Anna Dole\v{z}alov\'a}
\address{Department of Decision-Making Theory, Institute of Information Theory and Automation, Czech Academy of Sciences, Pod Vod\'arenskou
v\v{e}\v{z}\'i 4, 182~00 Prague 8, Czech Republic}
\email{dolezalova@utia.cas.cz}

\author[S. Hencl]{Stanislav Hencl}
\address{Department of Mathematical Analysis, Charles University,
So\-ko\-lovsk\'a 83, 186~00 Prague 8, Czech Republic}
\email{\tt hencl@karlin.mff.cuni.cz}

\thanks{D.C. and S.H. were supported by the grant GA\v{C}R P201/24-10505S. D.C. was also supported by the Ministry of Education, Youth and Sport of
	the Czech Republic grant number LL2105 CONTACT. A. D. was supported by the grant GA\v{C}R P202/23-04766S and by the Czech Academy of Sciences project PPLZ L100752451.}

\date{\today}

\begin{document}

\begin{abstract}
    {In this paper we prove the lower semicontinuity of a Neohookean-type energy for a model of Nonlinear Elasticity that allows, for the first time, for $p<n-1$. Our class of admissible deformations consists of weak limits of Sobolev $W^{1,p}$ homeomorphisms. We also introduce a model that allows for cavitations by studying weak limits of homeomorphisms that can open cavities at some points. 
		In this model we add the measure of the created surface to the energy functional and for this functional we again prove lower semicontinuity.}
\end{abstract}

\maketitle

\section{Introduction}

In this paper, we study classes of mappings that might serve as classes of deformations in Nonlinear Elasticity models. 
Let $\Omega\subseteq \rn$ be a bounded domain, i.e., a non-empty connected open set,
and let $f\colon\Omega\to\rn$ be a mapping with $J_f>0$ a.e. The natural physical deformation is the minimizer of the corresponding energy functional, typically of the form
$$
\int_{\Omega}W(Df(x))\; dx.
$$ 
We study the existence of such deformations using the direct method in the Calculus of Variations. As usual we study mappings that do not change orientation, i.e., $J_f>0$ a.e. We assume that the energy potential penalizes both large dilation,
$$
W(Df(x))\to\infty\text{ as }|Df(x)|\to \infty,
$$
and massive compression,
$$
W(Df(x))\to\infty\text{ as }J_f(x)\to 0+. 
$$

Following the pioneering paper of Ball \cite{Ball} we assume that our $W$ is polyconvex and that our map $f$ belongs to the Sobolev space $W^{1,p}(\Omega,\rn)$, i.e., that $W(Df)\geq C |Df|^p$ for some $p\in [1,\infty)$. We also need to impose some sort of injectivity a.e. on $f$, because of the physical requirement of ``the noninterpenetration of matter''.  There are several ways to achieve that. We can follow the ideas of Ball \cite{Ball2} and add some special terms to the energy functional based on the integrability of gradient minors (\cite{Ball}, \cite{Ball2}, \cite{FG}, \cite{Sv}, \cite{MTY}) or the distortion (\cite{Re}, \cite{IM}, \cite{HK}) and we obtain that the minimizing deformation (assuming nice homeomorphic boundary conditions) is actually a homeomorphism. Another way is to impose the additional so-called Ciarlet-Ne\v{c}as condition \cite{CN} which implies injectivity a.e. Both of these conditions require us to assume that $p\geq n$.

The condition $p>n$ is too restrictive, however, since in this context eligible deformations must necessarily be continuous, but some real physical deformations exhibit cavitations or even fractures. To model cavities for $W^{1,p}$-mappings, $p>n-1$, we can use the $(INV)$ condition which was introduced by M\"uller and Spector \cite{MS} (see also e.g.\  \cite{BHMC,HMC, HeMo11, MST}). Informally speaking, the $(INV)$ condition means that the ball $B(x,r)$ is mapped inside the image of the sphere $f(S(a,r))$ and the complement $\Omega\setminus \overline{B(x,r)}$ is mapped outside $f(S(a,r))$ (see the preliminaries for the formal definition).  From \cite{MS} we know that mappings in this class with $J_f>0$ a.e.\ are one-to-one a.e.\ and that this class is weakly closed and thus suitable for the direct method of the Calculus of Variations.

It is even possible to approach the limiting case $p=n-1$ and to define some version of the $(INV)$ condition there as shown by Conti and De Lellis \cite{CDL} 
(see also \cite{BHMCR}, \cite{BHMCR2}, \cite{DHM} and \cite{DHMo} for some recent work) and mappings in this class with $J_f>0$ a.e.\ are one-to-one a.e. too.  However, there is no way how to extend this notion of invertibility for $p<n-1$ as such maps are far from being continuous on $S(a,r)$. Therefore, for such maps we cannot define a degree and the notion of ``inside $f(S(a,r))$'' does not make sense.

The celebrated counterexample of Mal\'y \cite{M} gives us an example of diffeomorphisms $f_k$ that converge weakly to identity in $W^{1,p}$, $p<n-1$, but 
$\lim_{k\to \infty}\int J_{f_k}=0<\int J_{\operatorname{id}}$, i.e. this simple functional is not lower semicontinuous. The lower semicontinuity of functionals below the natural $W^{1,n}$ energy has attracted a lot of attention in the past and we refer the reader e.g.~to Ball and Murat 
\cite{BM}, Mal\'y \cite{M}, Dal Maso and Sbordone \cite{DMS} and  Celada and Dal Maso \cite{CM} for further information. However, in all of the positive results one has to assume that $p\geq n-1$.

In our paper we study the lower semicontinuity of a functional with $p$ below the threshold of $n-1$ for $n\geq 2$ and we are not aware of any positive results in this direction. As in \cite{CDL} we study Neohookean functionals of the type
\eqn{EnergyDef}
$$
E(f):=\int_{\Omega}\left(|Df(x)|^p+\varphi(J_f(x))\right)\; dx,
$$
where we assume the following natural growth assumptions 
\begin{equation}\label{varphi}
	\varphi \text{ is a positive convex function on }(0,\infty)\text{ with }
	\lim_{t\to 0^+}\varphi(t)=\infty,\ 
\end{equation}
and
\begin{equation}\label{varphi2}
	\lim_{t\to \infty}\frac{\varphi(t)}{t}=\infty.
\end{equation}

Our assumptions about the energy functional are quite minimal, but we have to assume something about the class of deformations to guarantee some sort of injectivity. Our class of deformations should contain homeomorphisms of finite energy and it has to be weakly closed so that we can apply the direct method of the Calculus of Variations. Hence, we suggest using the class of weak limits of Sobolev homeomorphism as our class of deformations. We also assume that our homeomorphisms satisfy the Lusin $(N)$ condition, i.e. that for any $E\subseteq \Omega$ with $|E|=0$ we have $|f(E)|=0$, as it is natural to assume that ``new material cannot be created from nothing''. This class was used also in \cite{DHM} and \cite{DHMo} for $p=n-1$, but there we needed to assume much more about the functional to obtain lower semicontinuity and the proofs there were quite technical. For any $p>\lfloor n/2\rfloor$ by modifying a recent result of Bouchala, Hencl and Zhu \cite{BHZ} we obtain that mappings in our class are injective a.e. which turns out to be crucial for us (see Theorem \ref{modBHZ}).

Given a fixed homeomorphism $f_0:\overline{\Omega}\to\rn$ satisfying the Lusin $(N)$ condition with $E(f_0)<\infty$ and $|f_0(\partial\Omega)|=0$ we denote 
$$
\begin{aligned}
\mathcal{H}_{f_0}^{1,p}=\bigl\{f\in W^{1,p}(\Omega,\rn):\ &f:\overline{\Omega}\to\rn\text{ is a homeomorphism, } f =f_0\text{ on }\partial\Omega\\
&\text{ and }f\text{ satisfies the Lusin }(N)\text{ condition}\bigr\}. \\
\end{aligned}
$$
Given $C>E(f_0)$ we define the class of weak limits
$$
\begin{aligned}
\overline{\mathcal{H}_{f_0}^{1,p}}^w=\bigl\{f:\Omega\to\rn:&\text{ there are }f_k\in \mathcal{H}_{f_0}^{1,p}\text{ with } E(f_k) \leq C\\
&\text{ so that }f_k\rightharpoonup f\text{ weakly in }W^{1,p}(\Omega, \rn)\bigr\}.\\
\end{aligned}
$$
Our main result is the following weak lower semicontinuity result with $p$ below the usual threshold of $n-1$.

\begin{thm}\label{LSCClosure}
Let $n\geq 2$, $p>\lfloor\frac{n}{2}\rfloor$, let $\Omega\subseteq \rn$ be a bounded domain and let $f_0 \in W^{1,p}(\Omega,\rn)$ be a homeomorphism from $\overline{\Omega}$ into $\rn$ which satisfies the Lusin $(N)$ condition, $E(f_0)<\infty$ and $|f_0(\partial\Omega)|=0$.  Let $f_k \in \overline{\mathcal{H}_{f_0}^{1,p}}^w$ satisfy $f_k\rightharpoonup f$ weakly in $W^{1,p}(\Omega,\rn)$, then $f\in \overline{\mathcal{H}_{f_0}^{1,p}}^w$ and 
\begin{equation}\label{LSCForClosure}
	E(f) \leq \liminf_{k\to\infty} E(f_k).
\end{equation}
It follows that $E$ attains its minimum on $\overline{\mathcal{H}_{f_0}^{1,p}}^w$.
\end{thm}

The main reason why the counterexample of Mal\'y \cite{M} does not contradict our result is the fact that the uniform boundedness of $E(f_k)$ in \eqref{EnergyDef} implies that small sets are mapped to small sets uniformly in $k$ (see Lemma \ref{l:reverse}  below) which is not the case for the diffeomorphisms $f_k$ in the counterexample in \cite{M}. 

The class of weak limits of Sobolev homeomorphisms was recently characterized in the planar case by Iwaniec and Onninen \cite{IO,IO2} and De Philippis and Pratelli \cite{DPP} (see also \cite{Ca}). The situation in higher dimension is much more difficult and we do not have a full characterization yet, but we know that mappings in this class do not change orientation \cite{HO} and are one-to-one a.e. \cite{BHZ} for $p$ big enough. We know that for $p>n-1$ a weak limit of Sobolev homeomorphisms satisfies the $(INV)$ condition as the $(INV)$ class is weakly closed. On the other hand already for $n=2$ there are mappings that satisfy the $(INV)$ condition but are not a weak limit of Sobolev homeomorphisms (see \cite{DPP}) and it is not clear if we want to have these mapping in our class of eligible deformations or not.

The motivation for models of Nonlinear Elasticity for $p>n-1$ introduced by M\"uller and Spector in \cite{MS} was partially to incorporate models that allow for cavitations. Unfortunately cavities cannot be created by weak limits of homeomorphisms when controlling the energy from \eqref{EnergyDef} as the sequence of limiting maps would necessarily violate the uniform measure estimate \eqref{reverse2} (see Lemma \ref{l:reverse} below) close to the point where the cavity is opened. We introduce a class of homeomorphisms that can open finitely many cavities and we study their weak limits (that can have countably many cavities). It was known already to Ball and Murat \cite[Counterexample 7.4]{BM} that weak lower semicontinuity may fail in this context if we do not assume something about the cavities. As usual in this theory (see e.g. \cite{CDL}, \cite{HeMo11}, \cite{HMC}) we assume that the total measure of the created surface is finite and we add the corresponding term to the energy functional. We show that the lower semicontinuity result holds also in this case for this new functional.

Let us denote a set of connected compact sets with positive measure and boundary of finite measure as
$$
\mathcal{K}(f_0(\Omega)):=\{K\subseteq f_0(\Omega):\ K\text{ is connected compact with } |K|>0 \text{ and }\haus^{n-1}(\partial K)<\infty\}. 
$$
Given a fixed homeomorphism $f_0:\overline{\Omega}\to\rn$ satisfying the Lusin $(N)$ condition with $E(f_0)<\infty$ and $|f_0(\partial\Omega)|=0$ we denote homeomorphisms which can have finitely many cavities as 
$$
\begin{aligned}
\mathcal{H}_{c,f_0}^{1,p}=\bigl\{f&\in W^{1,p}(\Omega,\rn):\ \text{there are }x_1,\hdots,x_m\in\Omega\text{ and disjoint }K_1,\hdots K_m\in\mathcal{K}(f_0(\Omega))\\
&\text{ so that }f:\overline{\Omega}\setminus\{x_1,\hdots,x_m\}\to\overline{f(\Omega)}\setminus(K_1\cup\hdots\cup K_m)\text{ is a homeomorphism},\\
&f =f_0\text{ on }\partial\Omega\text{ and }f\text{ satisfies the Lusin }(N)\text{ condition}\bigr\}. \\
\end{aligned}
$$

Note that if $n\geq 3$, $\Omega$ and $\Omega\setminus\{x_1,\dots,x_m\}$ have the same fundamental group. As $f$ is a homeomorphism on  $\Omega\setminus\{x_1,\dots,x_m\}$, the set $f(\Omega\setminus\{x_1,\dots,x_m\})$ has the same fundamental group. This imposes further conditions on $K_i$ (e.g. it cannot be a loop). In particular, for every two loops in $\Omega\setminus\{x_1,\dots,x_m\}$ which can be continuously deformed to each other, their images in $f(\Omega\setminus\{x_1,\dots,x_m\})$ can also be continuously deformed to each other.

Given $f\in \mathcal{H}_{c,f_0}^{1,p}$ we define its set of cavities as 
$$
A(f):=\bigcup_{i=1}^m K_i
$$
and we use an energy functional 
$$
E_c(f):=\int_{\Omega}\left(|Df(x)|^p+\varphi(J_f(x))\right)\; dx+a P(A(f),\mathbb{R}^n)
$$
for some $a>0$ where $P$ denotes the perimeter of the set (see Preliminaries). It is well known that $P(A,\mathbb{R}^n)=|D\chi_A|(\mathbb{R}^n)=\haus^{n-1}(\partial^* A)$ where $\partial^* A$ is the reduced 
boundary of $A$ (see e.g. \cite[Theorem 3.59]{AFP}) so the last term measures well the area of the created surface. 

Given $C>E(f_0)$ we define the class of weak limits
$$
\begin{aligned}
\overline{\mathcal{H}_{c,f_0}^{1,p}}^w=\bigl\{f:\Omega\to\rn:&\text{ there are }f_k\in \mathcal{H}_{c,f_0}^{1,p}\text{ with } E_c(f_k) \leq C\\
&\text{ so that }f_k\rightharpoonup f\text{ weakly in }W^{1,p}(\Omega, \rn)\bigr\}.\\
\end{aligned}
$$
Given $f\in \overline{\mathcal{H}_{c,f_0}^{1,p}}^w$ and $f_k\in\mathcal{H}_{c,f_0}^{1,p}$ that converge weakly to $f$ one can show that from the corresponding cavities $\chi_{A(f_k)}$ we can pick a subsequence which converges weak* in $BV$ to $\chi_A$ and this $A$ does not depend on the subsequence (see Theorem \ref{t:cavities_same} below), i.e., we can define $A(f):=A$ and functional $E_c$ is now defined also for $f$. 

\begin{thm}\label{LSCClosureCavity}
Let $n\geq 2$, $p>\lfloor\frac{n}{2}\rfloor$, let $\Omega\subseteq \rn$ be a bounded domain and let $f_0 \in W^{1,p}(\Omega,\rn)$ be a homeomorphism from $\overline{\Omega}$ into $\rn$ which satisfies the Lusin $(N)$ condition, $E_c(f_0)<\infty$ and $|f_0(\partial\Omega)|=0$.  Let $f_k \in \overline{\mathcal{H}_{c,f_0}^{1,p}}^w$ satisfy $f_k\rightharpoonup f$ in $W^{1,p}(\Omega,\rn)$, then $f\in \overline{\mathcal{H}_{c,f_0}^{1,p}}^w$ and 
\begin{equation}\label{LSCForClosureCavity}
	E_c(f) \leq \liminf_{k\to\infty} E_c(f_k).
\end{equation}
It follows that $E_c$ attains its minimum on $\overline{\mathcal{H}_{c,f_0}^{1,p}}^w$.
\end{thm}

This result is somewhat similar in spirit to the recent results by Barchiesi, Henao, Mora-Corral and Rodiac \cite{BHMCR3} and Mora-Corral and Mur-Callizo \cite{MCMC}. However, in those papers authors need to assume that $p\geq n-1$ as their assumptions essentially involve some properties of the distributional Jacobian. 
In our setting we allow for $p<n-1$ and hence we cannot even define the distributional Jacobian, so our arguments are in some sense simpler. On the other hand, our class of mappings is somewhat more narrow for $p\geq n-1$.

We know that homeomorphisms in the class $\mathcal{H}_{f_0}^{1,p}$ and $\mathcal{H}_{c,f_0}^{1,p}$ satisfy the Lusin $(N)$ condition and it would be great if the limiting mappings satisfies it as well. 
Unfortunately this is not the case even for strong limits of $W^{1,n-1}$ homeomorphisms as shown by the example in Theorem \ref{example} below. It means that in our proof of lower semicontinuity of energy functionals it is not enough to work with $|f(B)|$ (as it does not represent $\int _B J_f$) and instead we really have to work with terms like $|f(B\setminus N)|$ for some properly chosen null set $N$.

\prt{Theorem}
\begin{proclaim}\label{example}
For every $n\geq 3$ there is a continuous mapping $\f\colon[-1,1]^n\to [-1,1]^n$ with $J_{\f}>0$ a.e.\ which is a strong limit of Sobolev homeomorphisms $\f_k\in W^{1,n-1}([-1,1]^n,\rn)$ with $\f_k(x)=x$ for $x\in\partial[-1,1]^n$, $\f_k$ satisfy the Lusin $(N)$ condition and there is $\varphi$ satisfying \eqref{varphi} and \eqref{varphi2} so that 
\eqn{odkaz}
$$
\sup_k\int_{(0,1)^n}\varphi(J_{\f_k})<\infty
$$
such that $\f$ fails the Lusin $(N)$ condition. Moreover, the distributional Jacobian $\mathcal{J}_{\f}$ is equal to the pointwise Jacobian $J_{\f}$.  
\end{proclaim}

Note that it is very interesting that our map $f$ does not satisfy $(N)$ while the distributional Jacobian does not `see' that some new matter was created there. On the other hand, if we open a cavity then $\mathcal{J}_f$ detects this fact and measures its volume (see e.g. \cite[Remark 2.11 b)]{HK}). If we use the classical Ponomarev’s example (see e.g. \cite[Theorem 4.10]{HK}) of a homeomorphism which fails the Lusin $(N)$ condition then $\mathcal{J}_f$ detects the created matter and $\mathcal{J}_f$ is bounded from below by the restriction of a Hausdorff measure to the null Cantor set which is mapped to a Cantor set of positive measure.

In many theories (like e.g. the theory of mappings of finite distortion in \cite{HK} or models of Nonlinear Elasticity in \cite{HeMo11}, \cite{HMC} and so on) we first show or assume that $\mathcal{J}_f=J_f$ and then we prove that $f$ satisfies the $(N)$ condition. Conversely it is true that validity of $(N)$ in continuous $W^{1,n-1}$ mappings implies that $\mathcal{J}_f=J_f$ (see \cite{DHMS}). However, our example somewhat surprisingly shows that the validity of $(N)$ and $\mathcal{J}_f=J_f$ are not equivalent in this situation.

\section{Preliminaries}

By $f_{\rceil A}$ we denote the restriction of $f$ onto the set $A$. Let $a,b:X\to \er$ be a pair of mappings (typically sequences or functions); whenever we write that $a\approx b$, we mean that there is a number $C>1$ such that
$$
C^{-1}a(x)\leq b(x)\leq Ca(x) \text{ for all } x\in X.
$$

Using \eqref{varphi} and \eqref{varphi2} it follows from \cite[Lemma 2.9]{DHM} and \cite[Lemma 2.1]{DHMo} that small sets are mapped to small sets and big sets are mapped to big sets. Note that the measure bounds $\Phi$ and $\Psi$ depend only on the energy. 

\begin{lemma}\label{l:reverse}
    Given $C_1<\infty$ and $\varphi$ satisfying \eqref{varphi}, 
    there exist monotone functions $\Phi$,~$\Psi\colon (0,\infty)\to(0,\infty)$ with 
    $$
        \lim_{s\to 0^+}\Phi(s)=0\text{ and }\lim_{s\to 0^+}\Psi(s)=0
    $$
    such that: 
    Let $g\in W^{1,1}(\Omega,\rn)$ be a one-to-one mapping with $\int_{\Omega} \varphi(J_g)\le C_1$. 
    Then for each measurable set $A\subseteq \Omega$ we have
    \begin{equation}\label{reverse}
        \Phi(|A|)\le |g(A)|.
    \end{equation}
    If we moreover assume that the Lusin $(N)$ condition holds for $g$ and that \eqref{varphi2} holds, then also
    \begin{equation}\label{reverse2}
        |g(A)|\le \Psi(|A|).
    \end{equation} 
\end{lemma}

The following is \cite[Theorem~1.1]{BHZ}. Whenever we write $f_k\rightharpoonup f$, we mean that $f_k$ converge to $f$ weakly in the corresponding space. 
\begin{thm}\label{BHZthm} 
	Let $n\geq 2$, $\Omega\subseteq \rn$ be a domain and let $p>\left\lfloor\frac{n}{2}\right\rfloor$ for $n\geq 4$ or $p\geq 1$ for $n=2,3$. Let $f_k\in W^{1,p}(\Omega,\rn)$ be a sequence of homeomorphisms such that $f_k\rightharpoonup f$ weakly in $W^{1,p}(\Omega,\rn)$ and assume that $J_f>0$ a.e. 
	Then $f$ is injective a.e., i.e. there is a set $N\subseteq \Omega$ with $|N|=0$ such that $f|_{\Omega\setminus N}$ is injective. 
\end{thm}

The following is \cite[Theorem~1.1]{HO}.
\begin{thm}\label{HOthm}
	Let $n\geq 2$, $\Omega\subseteq \rn$ be a domain and let $p>\left\lfloor\frac{n}{2}\right\rfloor$ for $n\geq 4$ or $p\geq 1$ for $n=2,3$. Suppose that $f_k\in W^{1,p}(\Omega,\rn)$ is a sequence of homeomorphism such that $f_k\rightharpoonup f$ weakly in $W^{1,p}$ and further assume that $\det Df_k>0$ on a set of positive measure for each $k$. Then $\det Df \geq 0$ a.e. in $\Omega$.
\end{thm}

Both of these statements are generalized to our cases below.

\subsection{Condition for equiintegrability of Jacobians}

We need the following lemma from \cite[Lemma 6.4]{CDL} to show that a constructed sequence has equiintegrable Jacobians. 

\begin{lemma}\label{lemmaCDL}
    Let $\Omega\subseteq \rn$ be a domain and let $f_k:\Omega\to\rn$ be a sequence of bi-Lipschitz mappings. Assume that we have two conditions 
		$$
		\begin{aligned}
		(i) &\text{ for every }\epsilon>0\text{ there is }\delta>0\text{ so that for every measurable }E\subseteq \Omega\\
		&\text{ we have }|E|<\delta\Longrightarrow|f_k(E)|<\epsilon,\\
		(ii) &\text{ for every }\epsilon>0\text{ there is }\delta>0\text{ so that for every measurable }E\subseteq \rn\\
		&\text{ we have }|E|<\delta\Longrightarrow|f^{-1}_k(E)|<\epsilon.\\
		\end{aligned}
		$$
		Then there is a convex function $\varphi:\er\to [0,\infty)$ which satisfies \eqref{varphi} and \eqref{varphi2} such that
		$$
		\sup_k \int_{\Omega}\varphi(J_{f_k})<\infty. 
		$$
\end{lemma}

\subsection{Lower semicontinuity of functionals}

We need the following result of Eisen \cite{E}. 

\begin{thm}\label{LSC}
Let $\Omega\subseteq \rn$ be an open bounded set and let $W(x,u):\Omega\times \er^N\to[0,\infty)$ have the following properties:
$$
\begin{aligned}
(i)\ &f(\cdot,u): \Omega\to\er\text{ is measurable for every }u\in \er^N,\\
(ii)\ &f(x,\cdot): \Omega\to\er\text{ is convex for every }x\in \Omega,\\
\end{aligned}
$$
Let $u_k(x)$ be a sequence of functions such that $u_k(x)\rightharpoonup u(x)$ in $L^{1}(\Omega)$. 
Then 
$$
\int_{\Omega} W(x,u(x))\; dx \leq \liminf_{k\to\infty} \int_{\Omega} W(x,u_k(x))\; dx .
$$
\end{thm}

\subsection{(INV) condition}

Suppose that
$f\colon S(a,r) \to \rn$ is continuous.
Following \cite{MS} we define the {\it topological image} of $B(a,r)$ as
$$
f^T(B(a,r)):=\bigl\{y\in \rn\setminus f(S(a,r)):\ \deg(f,S(a,r),y)\neq 0\bigr\},
$$
where $\deg$ denotes the topological degree of the mapping. 
Denote
$$E(f,B(a,r)):=f^T(B(a,r))\cup f(S(a,r)).$$

\begin{definition}[(INV) condition]
We say that $f\colon\Omega\to\rn$ satisfies the condition (INV), provided that for every $a\in\Omega$ there exists an $\cL^1$-null set $N_a$ such that for all $r\in(0,\dist(a,\partial\Omega))\setminus N_a$ the mapping $f_{\rceil_{S(a,r)}}$ is continuous,
\begin{enumerate}[(i)]
	\item $f(x)\in E(f,B(a,r))$ for a.e.\ $x\in \overline{B(a,r)}$ and
	\item $f(x)\in\mathbb{R}^n\setminus f^T(B(a,r))$ for a.e.\ $x\in\Omega\setminus B(a,r)$.
\end{enumerate}
\end{definition}

\subsection{Sets of finite perimeter}

The symmetric difference of two sets is denoted as $A\triangle B=(A\setminus B)\cup(B\setminus A)$. 
	Let $A_k,A\subseteq \Omega \subseteq \rn$ be measurable sets throughout the following. When we write that $A_k\to A$, we mean $|A_k\triangle A|\to 0$ and we refer to this as convergence (in measure) in $\Omega$. If for each point $x\in \Omega$ we can find a neighborhood $G_x$ of $x$ such that $A_k\cap G_x\to A\cap G_x$, then we say that $A_k$ converge to $A$ locally in measure in $\Omega$. If $|A|<\infty$ then $A_k\to A$ in $\rn$ if and only if $\chi_{A_k} \to \chi_A$ in $L^1(\rn)$.

	For the definition of sets of finite perimeter see \cite[Definition~3.11]{AFP}. Let $\Omega \subseteq \rn$ be open. A set $A\subseteq \rn$ has finite perimeter in $\Omega$ if and only if $(\chi_A)_{\rceil \Omega} \in BV_{\loc}(\Omega)$ and $P(A,\Omega) = |D\chi_A|(\Omega)<\infty$ (see \cite[Theorem~3.36]{AFP}), where $|D\chi_A|$ denotes the total variation of the measure $D\chi_A$. A set $A$ has locally finite perimeter in $\Omega \subseteq \rn$ if for every $x\in \Omega$ there is a neighborhood $G_x$ of $x$ such that $P(A,G_x) = |D\chi_A|(G_x)<\infty$. The perimeter has the following lower semi-continuity property (see \cite[Proposition~3.38]{AFP}):
	\begin{prop}\label{maggi}
		Let $A_k\subseteq \rn$ be sets of locally finite perimeter in $\rn$ such that $A_k \to A$ locally in measure in $\rn$, then
		$$
			P(A,G) = |D\chi_A|(G) \leq \liminf_{k \to \infty}|D\chi_{A_k}|(G)  =\liminf_{k \to \infty} P(A_k,G)
		$$
		for every open $G\subseteq \rn$.
	\end{prop}

	For the following see \cite[Remark~3.37]{AFP} and \cite[Theorem~3.23]{AFP}. The local weak-star compactness of $BV$ is also conveyed to sets of finite perimeter in the following sense:
	\begin{thm}\label{vitana}
		Let $A_k\subseteq \rn$ be a sequence of sets of finite perimeter in $\rn$ with $\sup_k P(A_k,\rn)<\infty$, then there exists a subsequence $A_{k_m}$  and a set $A$ such that $A_{k_m}\to A$ locally in measure in $\rn$ and
		$$
		P(A,\Omega) \leq \liminf_{m\to\infty} P(A_{k_m}, \Omega)
		$$
		for all $\Omega \subseteq \rn$ open. Especially, if $A_{k}\subseteq B(0,R)$ for some $R>0$ then $A\subseteq B(0,R)$ and $A_{k_m}$ converge to $A$ in measure in $\rn$ globally.
	\end{thm}

The following isoperimetric inequality for sets of finite perimeter is \cite[Theorem~3.46]{AFP}.

\begin{thm}\label{isoperimetric}
Let $A\subseteq \rn$ be a bounded set of finite perimeter. 
Then
$$
|A|^{\frac{n-1}{n}}\leq C P(A,\mathbb{R}^n). 
$$
\end{thm}

\section{Weak limits of homeomorphisms}

We start this section by proving a generalization of Theorem \ref{BHZthm}. 

\begin{thm}\label{modBHZ}
	Let $n\geq 2$, $\Omega\subseteq \rn$ be a domain and let $p>\left\lfloor\frac{n}{2}\right\rfloor$ for $n\geq 4$ or $p\geq 1$ for $n=2,3$. Let $f_k\in W^{1,p}(\Omega,\rn)$ be a sequence of functions from $\mathcal{H}^{1,p}_{c,f_0}$ such that $f_k\rightharpoonup f$ weakly in $W^{1,p}(\Omega,\rn)$ and assume that $J_f\geq 0$ a.e. 
	
	Then $f$ is injective a.e. on $G:=\{x\in\Omega:J_f(x)>0\}$, i.e., there is a set $N\subseteq G$ with $|N|=0$ such that $f|_{G\setminus N}$ is injective. 
\end{thm}

\begin{proof}
In the planar case, the injectivity a.e. follows from the fact that weak limits of mappings satisfying the $(INV)$ condition also satisfy it. Our mappings $f_k$ satisfy the $(INV)$ condition as we know that $f_k$ equals to a homeomorphism $f_0$ on $\partial \Omega$, $f_k(\Omega)\subseteq f_0(\Omega)$ and $f_k$ itself is a homeomorphism on $\Omega\setminus \operatorname{Cav}(f_k)$.

Further we assume that $n\geq 3$ and then the proof is analogous to the proof of Theorem \ref{BHZthm}, so we mark here the changes necessary to modify the proof instead of rewriting it in its full extent.

The first change needed is in \cite[Lemma 3.1]{BHZ}. Let us denote $\tilde{G}$ the set of points from $G$ which are of density $1$ in $G$. One needs to replace the cubes $Q_k$ by $Q_k\cap \tilde{G}$ (and discard the cubes where $|Q_k\cap \tilde{G}|=0$). Then every two different points $x,y\in \tilde{G}$ can be distinguished by some set from the cover. The rest of the proof of the Lemma follows.

Now we focus on modifications in the proof of Theorem \ref{BHZthm} itself.

In Step 1, the sets $A$ and $B$ are subsets of $G$.

In Step 2a, we denote $\operatorname{Cav}(f_k)=\{x_1^k,\dots,x_{m_k}^k\}$ the set of points of discontinuity of $f_k$, and we may assume that $A$ and $B$ do not intersect $\operatorname{Cav}(f_k)$ for any $k$.

Note that in Step 2c, the authors assume implicitly that $r$ is small enough for $Q(a,r)$ and $Q(b,r)$ to be subsets of $\Omega$.

In Step 3a, we select $(x_2,r)$ such that neither the corresponding link $L_{x_2,r}$ nor the set
$$
S_{x_2,r}:=\left\{(x_1,x_2,x_3):\|(x_1,x_3)-\left(\frac{-2}{18},0\right)\|_\infty <r\right\},
$$
i.e., the square surface ``inside the link'', intersect $\cup_k \operatorname{Cav}(f_k)$. This is possible as the set is only countable. Analogously we select $x_3$ in Step 3b so that the link $\hat{L}_{x_3,r}$ and the square surface
$$
\hat{S}_{x_3,r}:=\left\{(x_1,x_2,x_3):\|(x_1,x_2)-\left(\frac{2}{18},0\right)\|_\infty <r\right\}
$$
does not intersect the set of points of discontinuity.

In Step 4, we fix $k>\max \{k_a,k_b\}$. Since $f_k(L_{x_2,r})$ and $f_k(\hat{L}_{x_3,r})$ are linked, the curve $f_k(\hat{L}_{x_3,r})$ has to intersect any surface which has $f_k(L_{x_2,r})$ as its boundary. However, $f_k(S_{x_2,r})$ is such surface (as $f_k$ is a homeomorphism there), which gives us the desired contradiction, as $f_k$ is a homeomorphism on
$$
L_{x_2,r}\cup S_{x_2,r} \cup \hat{L}_{x_3,r} \cup \hat{S}_{x_3,r}
$$
and therefore has to be injective there.

In Step 6, our surfaces $S_{x_2,r}$ and $\hat{S}_{x_3,r}$ are replaced by higher-dimensional surfaces.
\end{proof}

\begin{thm}\label{LSConHomThm}
	Let $p>\lfloor\frac{n}{2}\rfloor$, let $\Omega\subseteq \rn$ be a bounded domain and let $f_0$ be a homeomorphism from $\overline{\Omega}$ into $\rn$ which satisfies the Lusin $(N)$ condition, $E(f_0)<\infty$ and $|f_0(\partial\Omega)|=0$. It holds that
	$$
		E(f) \leq \liminf_{k\to\infty} E(f_k)
	$$
	for any sequence $f_k\in \mathcal{H}_{f_0}^{1,p}$ such that  
	$f_k \rightharpoonup f$ in $W^{1,p}(\Omega, \rn)$. 
\end{thm}

\begin{proof}
	\step{1}{Identify a function $J$ to which $\det Df_k$ ought to converge}{InitialSetup}

Without loss of generality we take some $f_k\in \mathcal{H}_{f_0}^{1,p}$ such that $E(f_k)$ is bounded and $f_k \rightharpoonup f$ in $W^{1,p}$. The weak convergence of $f_k$ in $W^{1,p}(\Omega,\rn)$ for $p>1$ implies that $f_k \to f$ strongly in $L^p(\Omega,\rn)$. Because $\det Df_k$ is bounded in some super-linear Orlicz space, a combination of the De la Vallée Poussin and the Dunford-Petis theorems implies (after taking a subsequence) that $\det Df_k$ converge weakly in $L^1(\Omega)$ to some function $J$. Then Theorem \ref{LSC} (for $u(x)=(u_1(x),u_2(x))=(Df(x),\det Df(x))$ and $W(x,u)=|u_1|^p+\varphi(u_2)$)  implies that
$$
	\int_{\Omega}|Df(x)|^p+ \varphi(J(x)) dx \leq \liminf_{k\to\infty} \int_{\Omega} |Df_k(x)|^p + \varphi(\det Df_k) dx.
$$ 
It now remains to prove that $\det Df = J$ almost everywhere on $\Omega$.

\step{2}{Show that $|f(B\setminus N)| = \int_B\det Df$}{AreaFormulaStep}

Before appealing to the area formula let us note that Theorem~\ref{HOthm} implies that $|\det Df| = \det Df$ a.e. on $\Omega$. 
Our sequence is converging weakly in $W^{1,p}$ and thus strongly in $L^p$ so we may assume without loss of generality that there is a zero measure set $N_0\subseteq \Omega$ such that (for a subsequence that we denote the same)
$$
f_k(x)\to f(x)\text{ pointwise on }\Omega\setminus N_0.
$$ 
It is well-known (see e.g. \cite{Ha}) that there exists a zero measure set $N_1\subseteq \Omega$ such that the area formula holds on $\Omega \setminus N_1$, i.e.,
$$
		\int_{\rn}\mathcal{H}^0(\{x\in \Omega\setminus N_1, f(x) = y\}) \, dy= \int_{\Omega\setminus N_1} \det Df(x)\, dx =  \int_{\Omega} \det Df(x)\, dx.
$$
By restricting this formula to the set $M =\{x\in \Omega \setminus N_1; \det Df(x) = 0\}$ we observe that $|f(M)|=0$. By Theorem~\ref{modBHZ} for $f_{\rceil \Omega\setminus M}$ we find a set $N_2\subseteq \Omega\setminus M$ of zero measure such that for every $x\in \Omega \setminus [M\cup N_2]$ we have 
$$
(\tilde{f}^{-1}\circ \tilde{f}(x))= \{x\}, \text{ where }
\tilde{f}:=f_{\rceil \Omega\setminus [M\cup N_2]}. 
$$
We denote 
\begin{equation}\label{ThisIsHowWeDefineN}
	N:= N_0\cup N_1\cup N_2. 
\end{equation}
Then for an arbitrary ball $B\subseteq \Omega$ we have $\mathcal{H}^0(\{x\in B\setminus N, f(x) = y\})= 1$ for all $y\in f(B\setminus N)\setminus f(M)$, giving
$$	\begin{aligned}
		|f(B\setminus N)| &= |f(B\setminus N)\setminus f(M)| =\int_{f(B \setminus N)\setminus f(M)} 1\, dy\\
		& =\int_{f(B \setminus N)} 1\, dy = \int_{B\setminus N} \det Df(x) \, dx=\int_{B} \det Df(x) \, dx.
	\end{aligned}
$$

On the other hand (by weak convergence in $L^1$) we have that
$$
	\int_{B}J(x) \, dx = \lim_{k\to\infty}\int_{B}\det Df_k(x) \, dx = \lim_{k\to\infty}|f_k(B)|.
$$
Therefore, we have $J=\det Df$ on the intersection of their Lebesgue points, if we prove that  
$$
|f(B\setminus N)| = \lim_{k\to\infty}|f_k(B)|\text{ for an arbitrary ball }B.
$$ 
The rest of the proof is dedicated to showing this fact.

\step{3}{Prove that $\limsup_{k\to\infty}|f_k(B)|\leq |f(B\setminus N)| $}{FirstEstimateStep}

We claim that $\limsup_{k\to\infty} \bigl|f_k(B)\bigr|\leq |f(B\setminus N)|$. For a contradiction assume that there is a positive $\epsilon$ such that 
$$
\limsup_{k\to\infty}\big|  f_k(B)\setminus f(B\setminus N)\big|> \epsilon.
$$ 
By approximation we find a compact set
$$
K \subseteq f_0(\Omega)\setminus f(B\setminus N)  \text{ such that } \bigl|f_0(\Omega)\setminus 
\bigl(K\cup f(B\setminus N)\bigr)\bigr|<\frac{\epsilon}{2} 
$$
It follows that 
$$
\limsup_{k\to\infty} \big| f_k(B)\cap K\big|>\frac{\epsilon}{2}.
$$
Let $k_m$ be a subsequence such that $\big| f_{k_m}(B)\cap K\big|>\epsilon/2$. Then, by Lemma~\ref{l:reverse}, the sets
$$
S_m := B\cap f^{-1}_{k_m}(K )  \ \text{ satisfy }  \ |S_{m}|\geq \delta_{\epsilon},
$$
and so it holds that
$$
S = \bigcap_{j=1}^{\infty}\bigcup_{m=j}^{\infty} S_m \subseteq B \  \text{ and } \ |S|\geq  \delta_{\epsilon}.
$$

For every $x\in S$ we have the existence of a further subsequence $k_{m_j(x)}$ such that $f_{k_{m_j(x)}} (x)\in K$ and therefore all accumulation points of this sequence also belong in $K$ as $K$ is compact. On the other hand for all $x\in S \setminus N$ (of which there are many since $|S|>0$) we have $f_{k_{m_j(x)}}(x)\xrightarrow{j\to\infty}f(x)\in f(B\setminus N)$. Because $K\cap f(B\setminus N) = \emptyset$ we have a contradiction.

\step{4}{Prove that $ \liminf_{k\to\infty}|f_k(B)|\geq |f(B\setminus N)|$}{SecondEstimateStep}

We now aim to prove that $\liminf_{k\to\infty}|f_k(B)|\geq |f(B\setminus N)| $. For a contradiction, let us assume that there is an $\epsilon>0$ such that there is a subsequence $k_m$ 
$$
	|f(B\setminus N)\setminus f_{k_m}(B)|>\epsilon \ \text{ for all }  m\in \en.
$$
Note that 
$f(\overline{\Omega})\subseteq f_0(\overline{\Omega})=f_k(\overline{\Omega})$ and $|f(\partial\Omega)|=0$, therefore $ f^{-1}_k$ is well-defined on $f(B\setminus N)$ and we can fix a compact set
\begin{equation}\label{hura}
K\subseteq f(B\setminus N) \setminus [f(M) \cup f(\partial\Omega)] \ \text { such that } \ | K \setminus f_{k_m}(B)|>\epsilon.
\end{equation}
By Lemma \ref{l:reverse} we have the existence of a $\delta_{\epsilon}>0$ such that
 $$
 |A_m|\geq \delta_{\epsilon}, \ \text{ where } \ A_m:=f_{k_m}^{-1}(K)\setminus B,
 $$
 and therefore it holds that
 $$
 A:=\bigcap_{l=1}^{\infty}\bigcup_{m=l}^{\infty} A_{m}  \ \ \text{ satisfies } \ \ |A|\geq\delta_{\epsilon}.
 $$
 Let $x\in A\setminus N$, then there is a subsequence $k_{m_{j(x)}}$ such that $f_{k_{m_{j(x)}}}(x)\in K$ for all $j$. The fact that $f_{k_{m_{j(x)}}}(x)\xrightarrow{j\to\infty}f(x)$ implies that $f(x)\in K$. Since $K\subseteq f(B\setminus N)\setminus f(M)$, we have $f = \tilde{f}$ on $\tilde{f}^{-1}(K)$ and then the above proven injectivity of $\tilde{f}$ on $\tilde{f}^{-1}(K)$ (see step~\ref{AreaFormulaStep}) implies that $x\in B\setminus N$. On the other hand, however, $A_m\cap B = \emptyset$ for all $m$ and so $x\in \Omega\setminus B$ which is a contradiction.
\end{proof}

In the proof above we have not only proved that $|f_k(B)|\to|f(B\setminus N)|$ but even that the measure of the symmetric difference tends to zero. 

\begin{corollary}\label{JackOfAllTrades}
	Let $f$ and $f_k$ be as in Theorem~\ref{LSConHomThm} with $E(f_k)$ bounded, then $\det D f_k \rightharpoonup \det Df$ in $L^1(\Omega)$. Moreover (by testing with $\chi_{\Omega}$) we have $\|\det D f_k\|_{L^1(\Omega)} = \|\det D f\|_{L^1(\Omega)} = |f_0(\Omega)|$ for all $k$ and 
	$$
	\lim_{k\to \infty} |f(B\setminus N) \Delta f_k(B)|=0
	$$
	for every ball $B\subseteq \Omega$, where the set $N$ is from \eqref{ThisIsHowWeDefineN}.
	It also follows, since $E(f)<\infty$, that $|M|=0$ and $f$ is therefore injective a.e. 
\end{corollary}

\begin{proof}[Proof of Theorem \ref{LSCClosure}]
Let $f\in W^{1,p}(\Omega, \rn)$ and $f_k \in \overline{\mathcal{H}_{f_0}^{1,p}}^w$ be such 	that $f_k \rightharpoonup f$ in $W^{1,p}(\Omega, \rn)$. Since $W^{1,p}(\Omega, \rn)$ is reflexive and separable (recall $p>\lfloor\tfrac{n}{2}\rfloor \geq 1$) we can find $L_i\in (W^{1,p}(\Omega, \rn))^*$ such that $\{L_i; i\in \en\}$ is dense. 
	By passing to a subsequence in $f_k$ if required, we can assume that
	$$
		|L_i(f_k-f)|<\frac{1}{k} \text{ for every }i\in\{1,\hdots,k\}. 
	$$
	
	For each $k\in \en$ we take a sequence $f_{k,m}\in\mathcal{H}_{f_0}^{1,p}$ such that $f_{k,m}\rightharpoonup f_k$ in $W^{1,p}(\Omega,\rn)$ and $E(f_{k,m})\leq C$ for all $k$ and $m$. 
	We may choose these sequences in such a way that they satisfy
	$$
	|L_i(f_{k,k}-f_k)|<\frac{1}{k} \text{ for every }i\in\{1,\hdots,k\}. 
	$$
	It follows that for every $i\in\en$ we have
	$$
	\lim_{k\to\infty} L_i(f_{k,k})=L_i(f). 
	$$
	Since $L_i$ are dense and $\int |Df_{k,k}|^p\leq E(f_{k,m})\leq C$ 
	this implies that $L(f_{k,k})\to L(f)$ for every $L\in (W^{1,p}(\Omega, \rn))^*$. Hence $f_{k,k}\rightharpoonup f$ in $W^{1,p}(\Omega, \rn)$ and therefore $f\in \overline{\mathcal{H}_{f_0}^{1,p}}^w$.

 	Corollary~\ref{JackOfAllTrades} implies for the sequence $f_{k,k}$ that $\det D f_k\rightharpoonup \det Df$ in $L^1(\Omega)$. Therefore we can apply Theorem~\ref{LSC} similarly as in the proof of Theorem~\ref{LSConHomThm} to conclude our proof.
	\end{proof}
	
\begin{corollary}\label{WeakJackOnClosure}
	Let $f$ and $f_k$ be as in Theorem~\ref{LSCClosure}, then $\det D f_k \rightharpoonup \det Df$ in $L^1(\Omega)$. Moreover $\|\det D f_k\|_{L^1(\Omega)} = \|\det D f\|_{L^1(\Omega)} = |f_0(\Omega)|$ for all $k$.
\end{corollary}

\section{Model that allows for cavitation}

\subsection{Proof of Theorem \ref{LSCClosureCavity}}

In this part, we use results from Proposition \ref{maggi} and Theorem \ref{vitana} to prove Theorem \ref{LSCClosureCavity}. The following two statements will help us to work with the cavity sets $A(f_k)$ and $A(f)$.

\begin{lemma}\label{lematko}
Let $A_k, A$ be bounded measurable sets. If  
$
A_k\to A,$
then there exists a subsequence $A_{k_\ell}$ such that 
$$
\Bigl| A \triangle \bigcap_{j\geq 1}^{\infty} \bigcup_{\ell \ge j} A_{k_\ell} \Bigr| = 0.
$$
\end{lemma}

\begin{proof}

We can select $k_\ell$ with 
$$
|A_{k_\ell} \triangle A| < 2^{-\ell}.
$$	
It follows that
$$
\Bigl|\bigcup_{\ell \ge j} (A_{k_\ell}\triangle A)\Bigr|<\sum_{\ell \ge j}2^{-\ell}=2^{-j+1}\overset{j\to\infty}{\to}0 
$$	
and hence (since the sets are nested)
$$
\Bigl|\bigcap_{j\geq 1}\bigcup_{\ell \ge j} (A_{k_\ell}\triangle A)\Bigr| = 0.
$$

We aim to estimate the size of the set
$$
V:=\Bigl(\bigcap_{j\geq 1}\bigcup_{\ell \ge j} A_{k_\ell}\Bigr)\triangle A
$$
and we do so by comparing it with the set
$$
W:=\bigcap_{j\geq 1}\bigcup_{\ell \ge j}\Bigl( A_{k_\ell}\triangle A\Bigr).
$$
We have
\begin{align*}
x\in V \text{ if and only if either } & x\in A \text { and there exists } \ell_0: x\notin A_{k_\ell} \text{ for all } \ell\geq \ell_0\\
\text{ or } & x\notin A \text{ and there exists } \{k_{\ell_m}\}: x\in A_{k_{\ell_m}} \text{ for all } m.
\end{align*}
Similarly,
\begin{align*}
x\in W \text{ if and only if either } & x\in A \text{ and there exists } \{k_{\ell_m}\}: x\notin A_{k_{\ell_m}} \text{ for all } m\\
\text{ or } & x\notin A \text{ and there exists } \{k_{\ell_m}\}: x\in A_{k_{\ell_m}} \text{ for all } m.
\end{align*}
This shows that $V\subseteq W$, and since we know that $|W|=0$, we obtain that $|V|=0$ as well.

\end{proof}

\begin{thm}\label{t:cavities_same}
Let $p>\lfloor\frac{n}{2}\rfloor$, let $\Omega\subseteq \rn$ be a bounded domain and let $f_0$ be a homeomorphism from $\overline{\Omega}$ into $\rn$ which satisfies the Lusin $(N)$ condition, $E(f_0)<\infty$ and $|f_0(\partial\Omega)|=0$. 
Let $f_k \in \mathcal{H}_{c,f_0}^{1,p}$ satisfy $E_c(f_k) \leq C$, 
$$
f_k \rightharpoonup f \text{ in } W^{1,p}(\Omega, \mathbb{R}^n) \quad \text{and } \quad f_k \to f \text{ a.e.}
$$
Assume we have two subsequences $f_{k_\ell}$ and $f_{k_m}$ such that
$$
A(f_{k_l})=:A_{k_\ell} \to A^1 \quad \text{and} \quad A(f_{k_m})=:A_{k_m} \to A^2.
$$
Then 
$$
|A^1 \triangle A^2| = 0.
$$
\end{thm}

\begin{proof}
For a contradiction, let us assume that $|A^2 \setminus A^1| > 0$.
 Select $D \subseteq A^2 \setminus A^1$:
$$
|D| > 0, \quad \exists \ell_0 \; \forall \ell \ge \ell_0 : D \cap A_{k_\ell} = \emptyset
$$
(that is possible because $ D \cap A^1 = \emptyset$ and we can assume -- again up to a subsequence -- that $|A^1 \triangle A_{k_\ell}|<2^{-\ell} |A^2\setminus A^1|$).
Given $\Phi$ and $\Psi$ from Lemma \ref{l:reverse}, we find 
$n \in \mathbb{N}$ so that
$$
\frac{2}{n} |D| < \Phi\Bigl( \frac{1}{4} \Psi^{-1}(|D|) \Bigr)
$$
and $G \supseteq D$ open such that
$$
|G| \leq \Bigl(1 + \frac{1}{n} \Bigr)|D|.
$$

Now we find $i \in \mathbb{N}$ such that 
\eqn{defdi}
$$
D_i := \{ y \in D : B(y, \tfrac{2}{i}) \subseteq G \}
$$ 
satisfies
\eqn{ddd}
$$
\Psi^{-1}(|D_i|) > \frac{1}{2} \Psi^{-1}(|D|) \quad \text{and} \quad \Bigl(1 - \frac{1}{n^2}\Bigr)|D_i| > \left(1 - \tfrac{1}{n}\right)|D|.
$$
Define $ F_{\ell} := f_{k_{\ell}}^{-1}(D_i) $. We know from Lemma \ref{l:reverse} that 
$$
|D_i| = |f_{k_{\ell}}(F_{\ell})| \leq \Psi(|F_{\ell}|),
$$
therefore $ |F_{\ell}| \geq \Psi^{-1}(|D_i|)$. Set $F := \bigcap_{j=1} \bigcup_{\ell \geq j} F_{\ell}$, then $ |F| \geq \Psi^{-1}(|D_i|)$. 
Denote 
\eqn{defk0}
$$
H^{k_0} := \{ x \in F : |f_{k}(x) - f_{k_0}(x)| < \tfrac{1}{i}  \text{ for every } k \geq k_0 \}
$$ 
and pick $k_0$ big enough
such that $|H^{k_0}| \geq \tfrac{1}{2}|F|$.

We know that for every $x \in H^{k_0}$ we can find $\ell$ so that $k_{\ell} \geq k_0$ with $x \in F_{\ell}$. Thus 
(using $ F_{\ell} = f_{k_{\ell}}^{-1}(D_i) $)  $f_{k_{\ell}}(x) \in D_i $ which together with \eqref{defdi} and \eqref{defk0}  implies that $f_k(x) \in G$ for all $k \geq k_0$.
Since $D_i\subseteq D \subseteq A^2$, we can use $A_{k_m}\to A^2$ to find $m_0$ such that $k_{m_0} \geq k_0$ and
$$
\text{ to find } \, \widetilde{D}_i \subseteq D_i \cap A_{k_{m_0}} \text{ with } |\widetilde{D}_i| > \Bigl(1 - \frac{1}{n^2}\Bigr)|D_i|.
$$
Now we know $f_{k_{m_0}}(H^{k_0}) \subseteq G$, at the same time $|f_{k_{m_0}}(H^{k_0}) \cap \widetilde{D}_i| = 0$ (because $\widetilde{D}_i\subseteq A_{k_{m_0}}$ and $A_{k_{m_0}}$ corresponds to cavities of $f_{k_{m_0}}$). Altogether we obtain using also Lemma \ref{l:reverse} and \eqref{ddd}
$$
\begin{aligned}
\Phi\left(\tfrac{1}{4} \Psi^{-1}(|D|)\right) &\leq \Phi\left(\tfrac{1}{2} \Psi^{-1}(|D_i|)\right)
\leq \Phi\left(\tfrac{1}{2} |F|\right)
\leq \Phi(|H^{k_0}|) \leq |f_{k_{m_0}}(H^{k_0})| \\
&\leq |G \setminus \widetilde{D}_i|
\leq |G \setminus D|+|D\setminus \widetilde{D}_i|
 < \tfrac{2}{n}|D|.\\
 \end{aligned}
$$
This gives us the desired contradiction.

\end{proof}

\begin{remark} Let $f_k \in \mathcal{H}_{c,f_0}^{1,p}$, $E_c(f_k) \leq C$ and $f_k \rightharpoonup f$. Assume two subsequences, $f_\ell$ and $f_m$, which converge pointwise a.e. As the pointwise limit of $f_{\ell_1}, f_{m_1}, f_{\ell_2}, f_{m_2}, \dots$ is $f$, it follows from Theorem \ref{t:cavities_same} that both $A(f_\ell)$ and $A(f_m)$ converge to the same set $A$ (up to a set of measure zero). Therefore for $f\in \overline{\mathcal{H}^{1,p}_{c,f_0}}^w$ we can denote this set by $A(f)$.
\end{remark}

The following observation tells us that even for the limiting mapping $f\in \overline{\mathcal{H}^{1,p}_{c,f_0}}^w$ cavities are disjoint from the image of the body as expected. 

\begin{lemma}\label{nutne}
Let $p>\lfloor\frac{n}{2}\rfloor$, let $\Omega\subseteq \rn$ be a bounded domain and let $f_0$ be a homeomorphism from $\overline{\Omega}$ into $\rn$ which satisfies the Lusin $(N)$ condition, $E(f_0)<\infty$ and $|f_0(\partial\Omega)|=0$. 
Let $f\in \overline{\mathcal{H}^{1,p}_{c,f_0}}^w$, then $f(\Omega\setminus N)$ and $A(f)$ are essentially disjoint.
\end{lemma}

\begin{proof} 
Assume for contrary that there is a compact set 
$$
K_0\subseteq f(\Omega\setminus N)\cap A(f)\text{ with }|K_0|>0. 
$$
Fix $\epsilon>0$. We know that $A(f_k)\to A(f)$ in measure and hence we can can assume that (for a subsequence that we still denote the same)
$$
\sum_k|A(f)\triangle A(f_k)|<\frac{|K_0|}{2}. 
$$
It follows that we can find a compact set 
$$
K\subseteq f(\Omega\setminus N)\cap A(f)\text{ with }|K|\geq  \frac{|K_0|}{2}>0\text{ such that }K\subset A(f_k)\text{ for every }k. 
$$
Analogously to the proof of Step 2 of Theorem \ref{LSConHomThm} we obtain that 
$$
\int_A J_f=|f(A\setminus N)|
$$
for an arbitrary measurable set $A$ and thus $|f^{-1}(K)|>0$.

We fix $\eta>0$ small enough so that the set 
\eqn{uuuu}
$$
K_{\eta}:=K+B(0,\eta)\text{ satisfies }|K_{\eta}\setminus K|<\Phi\Bigl( \frac{|f^{-1}(K)|}{2}\Bigr),  
$$
where $\Phi$ is from Lemma \ref{l:reverse}.
For almost every $x\in f^{-1}(K)$ we know that $f_k(x)\to f(x)\in K$. It follows that there is $k_x$ such that for $k\geq k_x$ we have $f_k(x)\in K_{\eta}$. We can fix $k_0$ such that for 
$$
E:=\bigl\{x\in f^{-1}(K): k_x\leq k_0\bigr\}\text{ we have } |E|> \frac{|f^{-1}(K)|}{2}. 
$$
From Lemma \ref{l:reverse} we obtain that 
\eqn{uuu}
$$
|f_{k_0}(E)|\geq \Phi\Bigl( \frac{|f^{-1}(K)|}{2}\Bigr).
$$
On the other hand $k_x\leq k_0$ for every $x\in E$ which implies that $f_{k_0}(x)\in K_{\eta}$ and since $f_k(x)\notin A(f_k)$ (by the definition of the cavity) and $K\subseteq A(f_k)$ for every $k$ we obtain that
$$
f_{k_0}(x)\in K_{\eta}\setminus K\text{ and hence }f_{k_0}(E)\subseteq K_{\eta}\setminus K.
$$
Now \eqref{uuu} and \eqref{uuuu} give us the desired contradiction. 
\end{proof}

Let us now return to Theorem \ref{LSCClosureCavity}. We need the following result.

\begin{thm}\label{LSConHomThm_Cav}
	Let $p>\lfloor\frac{n}{2}\rfloor$, let $\Omega\subseteq \rn$ be a bounded domain and let $f_0$ be a homeomorphism from $\overline{\Omega}$ into $\rn$ which satisfies the Lusin $(N)$ condition, $E(f_0)<\infty$ and $|f_0(\partial\Omega)|=0$. It holds that
	$$
		E(f) \leq \liminf_{k\to\infty} E(f_k)
	$$
	for any sequence $f_k\in \mathcal{H}_{c,f_0}^{1,p}$ such that $f_k \rightharpoonup f$ in $W^{1,p}(\Omega, \rn)$. 
\end{thm}

\begin{proof}
We proceed similarly to the proof of Theorem \ref{LSConHomThm}. The first three steps do not require any modification except that we use a version of Theorem \ref{HOthm} for homeomorphisms with finitely many cavities. (Since those mappings are sense-preserving, their proof works also for our case.) In the last step, we need to show that 
$$
\liminf_{k\to\infty} |f_k(B)|\geq |f(B\setminus N)|.
$$
We again fix a subsequence $k_m$ such that
$$
|f(B\setminus N)\setminus f_{k_m}(B)|>\varepsilon.
$$
Moreover, we know that $\sup_k P(A(f_{k}),\mathbb{R}^n)<\infty$ and thus we can find a subsequence of $A(f_k)$ which converges in measure (see Theorem \ref{vitana}). It follows that for our subsequence we can moreover assume that (analogously to the proof of Lemma \ref{lematko})
\eqn{eee}
$$ 
|A(f) \triangle \bigcup_j A(f_{k_j})| <\varepsilon/2. 
$$
Now we find a compact set $K$ such that
$$
K\subseteq f(B\setminus N)\setminus f(M) \text{ and } |K\setminus f_{k_m}(B)|>\varepsilon.
$$
By Lemma \ref{nutne} we know that $|K\cap A(f)|=0$ and hence we can use \eqref{eee} to obtain that
\eqn{eeee}
$$
K\setminus \Bigl(f_{k_m}(B)\cup\bigcup_j A(f_{k_j})\Bigr) \text{ has measure bigger that }\frac{\epsilon}{2}. 
$$
Define
$$
A_m:= f^{-1}_{k_m}\Bigl(K\setminus \bigcup_j A(f_{k_j})\Bigr)\setminus B.
$$
Since we avoided the set of cavities, $A_m$ is well-defined and due to \eqref{eeee} and Lemma \ref{l:reverse} there exists $\delta_\varepsilon$ such that $|A_m|\geq \delta_\varepsilon$. Therefore also
$
A:=\cap\cup A_m
$
satisfies $|A|\geq \delta_\varepsilon$.
Now we can obtain the desired contradiction as before, by showing that for $x\in A\setminus N$ we have $f(x)\in K$, and therefore $x\in B$, contradicting $A\cap B=\emptyset$.
\end{proof}

\begin{proof}[Proof of Theorem \ref{LSCClosureCavity}]
Along the same lines as the proof of Theorem \ref{LSCClosure} we can prove that 
$$
E(f)\leq \liminf_{k\to\infty} E(f_k).
$$ 
It remains to show similar inequality for the last term of $E_c$, i.e., for $P(A(f),\mathbb{R}^n)$. From (any subsequence of) the bounded sequence $A(f_k)$ we can select a subsequence converging in measure (see Theorem \ref{vitana}) so that
$$
P(A(f),\mathbb{R}^n)\leq \liminf_{k\to\infty} P(A(f_k),\mathbb{R}^n).
$$

Note that we again can show that $|f(B\setminus N)\triangle f_k(B)|\to 0$.
\end{proof}

\begin{corollary}
Let $f$ and $f_k$ be as in Theorem \ref{LSCClosureCavity}, then $\det Df_k\rightharpoonup \det Df$ in $L^1(\Omega)$ and $|f_0(\Omega)| = ||\det Df_k||_{L^1(\Omega)} + |A(f_k)| = ||\det Df||_{L^1(\Omega)} + |A(f)|$.
\end{corollary}

\subsection{Topological image}

In this part, we look at the analogy of the notion of topological image. In general, for $p<n-1$ we cannot use the standard definition via the topological degree, as the mapping can behave too discontinuously on the boundary of a sphere. Instead, we define the topological image of a ball $B$ under $f\in \overline{\mathcal{H}^{1,p}_{c,f_0}}^w$ as
$$
f_T(B):= f(B\setminus N)\cup A(f_{\rceil B}),
$$
where $N\subseteq \Omega$ is the set such that $f_k$ converges pointwise and $f$ satisfies the area formula on $\Omega\setminus N$ and $f_{\rceil B}$ denotes the restriction of $f$ onto $B$. In the following we show that this notion satisfies some natural properties.

To define the topological image of a point $x\in \Omega$, we need to know that for $r_1<r_2$ we have 
$$
f_T(B(x,r_1))\subseteq f_T(B(x,r_2)).
$$
The inclusion of the first term holds trivially. For the second term, pick $f_k\rightharpoonup f$ such that $f_k\to f$ pointwise a.e. Then for all $k\in\mathbb{N}$ we have
$$
A((f_k)_{\rceil B(x,r_1)})\subseteq A((f_k)_{\rceil B(x,r_2)}),
$$
and therefore (up to a set of measure zero)
$$
A(f_{\rceil B(x,r_1)})\subseteq A(f_{\rceil B(x,r_2)}).
$$ 
This allows us to set 
$$
f_T(x):= \bigcap_{r>0, B(x,r)\subseteq \Omega}f_T(B(x,r)).
$$
(Note that for a.e. $x\in\Omega$ we have $f(x)\in f(B(x,r)\setminus N)$, and so $f(x)\in f_T(x)$.) In case $x\in\partial\Omega$, we can define the topological image as well, replacing the ball $B$ by $B\cap\overline{\Omega}$ in the definitions above (we usually do not distinguish this case in the later). We say that $f$ opens a cavity at $x$ if $|f_T(x)|>0$, and we denote by $\operatorname{Cav}(f)$ the set of all points where $f$ opens a cavity, i.e.,
$$
\operatorname{Cav}(f):=\bigl\{x\in\Omega:\ |f_T(x)|>0\bigr\}. 
$$ 
\begin{remark}
Note that for $p\leq n-1$ this definition does not coincide with the classical definition of topological image as the set $\{y: \deg(y,f,B)\neq 0\}$, even if $f$ is continuous on $\partial B$. This can be observed e.g. in the counterexample of Conti and De Lellis \cite{CDL} (see also \cite{DHM}) of a mapping where the $(INV)$ condition fails. In this example the image of a ball (which agrees with our notion of topological image) goes outside of the set $\{y: \deg(y,f,B)\neq 0\}$. 

For $p>n-1$ one could expect that the definitions of topological image are the same. 
\end{remark}

\begin{lemma}\label{l:disjoint_cav} 
Let $p>\lfloor\frac{n}{2}\rfloor$, let $\Omega\subseteq \rn$ be a bounded domain and let $f_0$ be a homeomorphism from $\overline{\Omega}$ into $\rn$ which satisfies the Lusin $(N)$ condition, $E(f_0)<\infty$ and $|f_0(\partial\Omega)|=0$. 
Let $x_1, x_2 \in \Omega$ and $f \in \overline{\mathcal{H}^{1,p}_{c,f_0}}^w$.
Then $|f_T(x_1) \cap f_T(x_2)| = 0$.
\end{lemma}

\begin{proof} 
Let $f_k\in \mathcal{H}^{1,p}_{c,f_0}$ be mappings that converge to $f$ weakly. 
For an arbitrary $\varepsilon > 0$ and $0 < r < \tfrac{1}{2} |x_1 - x_2|$ there exists $k_0$ such that for $k\geq k_0$ (see Theorem \ref{t:cavities_same})
$$
\left| A((f_{k})_{ \rceil B(x_i, r)}) \bigtriangleup A(f_{ \rceil B(x_i, r)}) \right| < \varepsilon.
$$
We know that $A((f_k )_{ \rceil B(x_i, r)})$ are mutually disjoint, therefore
$$
\left| A(f_{B( \rceil x_1, r)}) \cap A(f_{ \rceil B(x_2, r)}) \right| < 2\varepsilon.
$$
Since $|f(B(x_i, r)\setminus N)| \xrightarrow{r \to 0+	} 0$, we can find $r_0$ such that for every $ r < r_0$ we have
$$
|f_T(B(x_1, r)) \cap f_T(B(x_2, r))| < 3\varepsilon.
$$
As $f_T(x_i) \subseteq f_T(B(x_i, r))$, we are done.
\end{proof}

\begin{corollary}
From Lemma \ref{l:disjoint_cav} we immediately obtain that $f\in \overline{\mathcal{H}^{1,p}_{c,f_0}}^w$ can open only countably many cavities.
\end{corollary}

For $f\in \mathcal{H}^{1,p}_{c,f_0}$ there is a pairing of the points of discontinuity $x_1,\dots,x_m$ and the resulting cavities $K_1,\dots,K_m$. In particular, since $f$ satisfies the Lusin $(N)$ condition,
\begin{equation}\label{EqTopImageCav}
A(f)=\bigcup_i K_i = \bigcup_{x_i \in \operatorname{Cav}(f)} f_T(x_i)
\end{equation}
up to a set of measure zero.
We would like to have an analogous result for weak limits of such mappings, connecting the topological images of points where cavities are open and the limiting set $A(f)$ (which is just the limit of the corresponding $A(f_k)$). The following theorem shows that indeed the analogy of \eqref{EqTopImageCav} holds.

\begin{thm} 
Let $p>\lfloor\frac{n}{2}\rfloor$, let $\Omega\subseteq \rn$ be a bounded domain and let $f_0$ be a homeomorphism from $\overline{\Omega}$ into $\rn$ which satisfies the Lusin $(N)$ condition, $E(f_0)<\infty$ and $|f_0(\partial\Omega)|=0$. 
Let $f\in \overline{\mathcal{H}^{1,p}_{c,f_0}}^w$, then 
$$
|A(f)\triangle \bigcup_{x\in \operatorname{Cav}(f)} f_T(x)|=0.
$$
\end{thm}

\begin{proof}
For our $f$ we find $f_k\in \mathcal{H}^{1,p}_{c,f_0}$ with $f_k \rightharpoonup f$ in $W^{1,p}$ and pointwise a.e. and $\chi_{A(f_k)}\to \chi_{A(f)}$ weakly* in $BV$ and strongly in $L^1$. Let $x_0\in \operatorname{Cav}(f)$, i.e., a point where $f$ opens a cavity.

We know that 
$$
f_T(x_0)=\bigcap_{r>0} (f(B(x,r)\setminus N)\cap A(f_{\rceil B(x,r)})).
$$ 
Since $f$ satisfies the Lusin $(N)$ condition on $\Omega\setminus N$ (see the definition of the set $N$ in the proof of Theorem \ref{LSConHomThm}), we have that 
$$
|f(B(x,r)\setminus N)|\to 0\text{ as }r\to 0+.
$$ 
Therefore $f_T(x_0)\subseteq A(f)$  up to a set of measure zero. 

We now show the opposite inclusion.
To do so, we claim that given $\varepsilon>0$, the set 
$$
T_\varepsilon:=\bigl\{y\in A(f): \text{ there exists a subsequence } k_\ell \text{ such that } y\in K_{i(k_\ell)}^{k_\ell}, |K_{i(k_\ell)}^{k_\ell}|>\varepsilon\bigr\}
$$
is of size at least 
\eqn{size}
$$
|T_\varepsilon|\geq |A(f)|-C\varepsilon^{\tfrac{1}{n}}, \text{ where }C\text{ is independent of }\varepsilon.
$$

Assuming that, we know that for a.e. $y_0\in A(f)$ we can find $\varepsilon$ such that $y_0\in T_\varepsilon$. Therefore we have a sequence of sets $K_{i(k_\ell)}^{k_\ell}$ which are of size at least $\varepsilon$ and containing $y_0$. For each $f_{k_\ell}$, the cavity $K_{i(k_\ell)}^{k_\ell}$ corresponds to a point $x_{i(k_\ell)}^{k_\ell}\in \operatorname{Cav}(f_{k_\ell})$. This sequence of points has to have at least one accumulation point $x_0\in\overline{\Omega}$ (so without loss of generality we assume that $x_{i(k_\ell)}^{k_\ell}\to x_0$). 

For that $x_0$ we know that $f_T(x_0)\supseteq \bigcap_{r>0} A(f_{\rceil B(x_0,r)})$. For any $r>0$ we can find $k_{\ell_0}$ such that $x_{i(k_\ell)}^{k_\ell}\in B(x_0,r)$ for every $k_{\ell}\geq k_{\ell_0}$. Therefore 
$$
|A((f_{k_\ell})_{\rceil B(x_0,r)})|\geq \varepsilon,
$$
and so (from the convergence of sets) 
$$
|A(f_{\rceil B(x_0,r)})|\geq \varepsilon.
$$
Since for $r_1<r_2$ we have 
$$
A(f_{\rceil B(x_0,r_1)})\subseteq A(f_{\rceil B(x_0,r_2)})
$$
up to a set of measure zero, we can conclude that $|f_T(x_0)|\geq \varepsilon$ and therefore $x_0\in \operatorname{Cav}(f)$. At the same time, 
$$
y_0\in K_{i(k_\ell)}^{k_\ell}\subseteq A((f_{k_\ell})_{\rceil B(x_0,r)}),
$$
and so $y_0\in f_T(x_0)$.

It remains to prove \eqref{size}. Fix $\varepsilon>0$ and denote 
$$
\mathcal{K}^k_+ :=\{K_i^k \text{ cavities of } f_k: |K_i^k|\geq\varepsilon\}, A(f_k)_+:=\bigcup_{K \in \mathcal{K}^k_+} K,
$$
and
$$
\mathcal{K}^k_- :=\{K_i^k \text{ cavities of } f_k: |K_i^k|<2\varepsilon\}, A(f_k)_-:=\bigcup_{K \in \mathcal{K}^k_-} K.
$$
Since $A(f_k)=A(f_k)_+\cup A(f_k)_-$, we know that using ideas of Lemma \ref{lematko} we can assume that for a subsequence (that we still denote as $f_k$) we have
$$
A(f)=\left(\bigcap_{j=1}^{\infty} \bigcup_{k=j}^{\infty} A(f_k)_+\right) \cup \left(\bigcap_{j=1}^{\infty} \bigcup_{k=j}^{\infty} A(f_k)_-\right).
$$
Now using Theorem \ref{isoperimetric} and $\sup_k E_c(f_k)<\infty$ we know that
$$
|A(f_k)_-|=\sum_{K \in \mathcal{K}^k_-} |K| \leq C\varepsilon^\frac{1}{n} \sum_{K \in \mathcal{K}^k_-} |K|^\frac{n-1}{n}\leq C \varepsilon^\frac{1}{n}  \sum_{K \in \mathcal{K}^k_-} P(K, \mathbb{R}^n) \leq C_0\varepsilon^\frac{1}{n}.
$$ 
We know that $\chi_{A(f_k)}$ converge to $\chi_{A(f)}$ in $L^1$ and hence we can assume that for all $k$ we have $|A(f)\triangle A(f_k)|<C_0\varepsilon^\frac{1}{n}$. 
Therefore
$$
|A(f_k)_+|\geq |A(f_k)| - |A(f_k)_-|\geq |A(f)|- 2C_0\varepsilon^\frac{1}{n}.
$$
Since $T_\varepsilon = \bigcap \bigcup A(f_k)_+$, we obtain \eqref{size}.
\end{proof}

\section{Counterexample to the Lusin (N) condition}

In our proof of Theorem \ref{example} we use some ideas and results from \cite{BHM}. Unfortunately the two examples from \cite{BHM} cannot be used directly and we have to study other properties 
like the integrability of $\int \varphi(J_f)$ and properties of the distributional Jacobian that were not studied there. We include most of the details so that the reader can follow even though the  definition of various Cantor type sets and mappings between them is the same as in \cite{BHM}. 

Let us provide a short overview of the mapping we wish to construct. We start with a Cantor set of positive measure called $C_A$ and for each point $y\in C_A$ we find a curve $q_y$ (a ``tentacle'') with the following properties:
\begin{itemize}
	\item each $q_y$ is a Lipschitz curve of nonzero finite length,
	\item the curves $q_y$ are pairwise disjoint,
	\item it holds that $|\bigcup_{y\in C_A}q_y \setminus C_A| = 0$.
\end{itemize}
We construct a mapping $\tilde{f} \in W^{1,n-1}$ as the pointwise and $W^{1,n-1}$-limit of bi-Lipschitz mappings such that 
\begin{itemize}
	\item $\tilde{f}$ equals the identity on $C_A$,
	\item $\tilde{f}$ equals the identity everywhere outside a small neighborhood of $\bigcup_{y\in C_A}q_y$,
	\item $\tilde{f}(q_y) = \{y\}$ for each $y\in C_{A}$.
\end{itemize}
It then readily follows that our map  does not satisfy the Lusin $(N)$ condition as $\f\big(\bigcup_y q_y\setminus C_A\big) = C_A$ and $|C_A|>0$. With some work, and the help of Lemma~\ref{lemmaCDL}, we find a function $\varphi$ such that $E(\tilde{f})$ is finite for the corresponding energy functional $E$ from \eqref{EnergyDef}.

\subsection{Construction of Cantor sets and homeomorphisms between them}\label{ssec:CS}

Following \cite[Section 4.3]{HK} we consider a Cantor-set construction in $[-1,1]^n$.

Denote the cube with center at
$z$ and side length $2r$ by
$Q(z,r) = [z_1 - r, z_1+r] \times \dots \times [z_n - r, z_n+r] $.
Let
$\mV= \{-1,1\}^n$ be the set of
$2^n$
vertices of the cube
$[-1, 1]^n \subseteq \mR^n$
and
$\mV^k = \mV \times \cdots \times \mV$,
$k\in \en$.
Consider a sequence
$\{\alpha_k\}_{k=0}^{\infty}$
such that $\alpha_k\approx \alpha_{k+1}$
$$
1=\alpha_0\geq \alpha_1 \geq \dots >0,
$$
 and set
$$
r_k := 2^{-k} \alpha_k\text{ and }r'_k:=2^{-k}\alpha_{k-1}=r_{k-1}/2.
$$
Set $z_0=0$, then $Q(z_0,r_0) = (-1,1)^n$
and we proceed by induction.
For
$$\mv(k)=(v_1, \dots , v_k) \in \mV^k$$
we define 
\begin{equation*}
\begin{aligned}
    & z_{\mv(k)} = z_{(v_1, \dots , v_{k-1})} + \frac{1}{2} r_{k-1}v_k = 
    \frac{1}{2}\sum_{j=1}^{k} r_{j-1}v_j.
\end{aligned}
\end{equation*}
Around these centers we we place a smaller and a bigger cube
\begin{equation*}
\begin{aligned}
    & Q_{\mv(k)} = Q(z_{\mv(k)}, r_k) \quad \text{and} \quad  Q'_{\mv(k)} = Q(z_{\mv(k)}, r'_k)
\end{aligned}
\end{equation*}
(see Figure~\ref{pic:CantorSet}). Informally speaking, we always divide the cube $Q_{\mv(k)}$ into $2^n$ cubes of half the side length and then we inscribe a smaller cube into each of them. Here $v(j)$ can be understood as the direction from the center of the $(j-1)$-th cube in which we select the smaller cube in the $j$-th generation.

\begin{figure}[h]

\begin{tikzpicture}[scale=1]
\newcommand\Square[1]{+(-#1,-#1) rectangle +(#1,#1)}
 
  \foreach \x in {-1,1}
    \foreach \y in {-1, 1}
    \draw[fill=lightgray] (\x,\y) \Square{1}; 
 
   \foreach \x in {-1,1}
    \foreach \y in {-1, 1}
    \draw[fill= white] (\x,\y) \Square{0.8}; 

    \node[above] at (-1.25, -0.9) {$Q_{\mv(1)}$};

    \node[above] at (-2.5, -2) {$Q'_{\mv(1)}$};

  \foreach \x in {-1,1}
    \foreach \y in {-1, 1}
    \draw (\x+5.5,\y) \Square{1}; 
 
   \foreach \x in {-1.4,-0.6,0.6, 1.4}
    \foreach \y in {-1.4,-0.6,0.6, 1.4}
    \draw[fill=lightgray] (\x+5.5,\y) \Square{0.4};
    
   \foreach \x in {-1.4,-0.6,0.6, 1.4}
    \foreach \y in {-1.4,-0.6,0.6, 1.4}
    \draw[fill=white] (\x+5.5,\y) \Square{0.3};

\end{tikzpicture}
\caption{First two generations in the construction. The cubes $Q_{\mv(k)}$ are white, the gray ``frames'' are the cubical annuli $Q'_{\mv(k)}\setminus Q_{\mv(k)}$.}
\label{pic:CantorSet}
\end{figure}
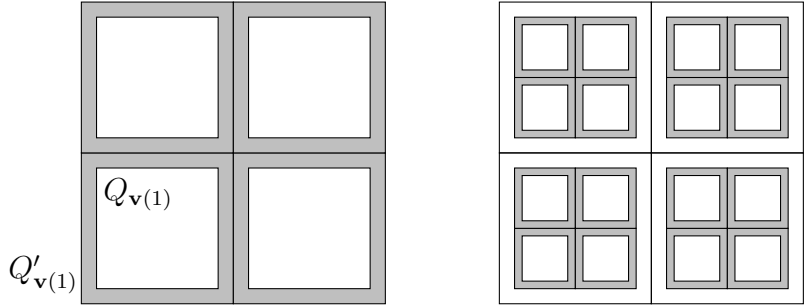

For a sequence
$A=\{\alpha_k\}_{k=0}^{\infty}$
the resulting Cantor set
$$
    C_A := \bigcap_{k=1}^{\infty} \bigcup_{\mv(k)\in \mV^{k}} Q_{\mv(k)} =  \Big\{\frac{1}{2}\sum_{j=1}^{\infty} r_{j-1}v_j; v\in \mV^{\en} \Big\}
$$
is a product of $n$ Cantor sets $\cC_\alpha$ in $\mR$.

In the $k$-th generation of the construction, the number of cubes in
$\{Q_{\mv(k)}: \mv(k) \in \mV^k\}$
is
$2^{nk}$.
Hence,
\eqn{size_cantor}
$$
|C_A|= \lim_{k\to\infty} 2^{nk}(2 r_k)^n =\lim_{k\to\infty} 2^{nk}(2 \alpha_k 2^{-k})^n = \lim_{k\to\infty}2^n \alpha_k^n.
$$

Consider two sequences
$A=\{\alpha_k\}_{k=0}^{\infty}$
and
$B=\{\beta_k\}_{k=0}^{\infty}$ satisfying the conditions from above
and the two corresponding Cantor sets
$C_A$ and $C_B$
 designed as described.
We also denote 
\begin{equation*}
\begin{aligned}
    & \tilde{r}_k = 2^{-k}\beta_k, \quad \tilde{r}_k' = 2^{-k}\beta_{k-1},\\
    & \tilde{z}_{\mv(k)} = \tilde{z}_{[v_1, \dots , v_{k-1}]} + \frac{1}{2} \tilde{r}_{k-1}v_k =
    \frac{1}{2}\sum_{j=1}^{k} \tilde{r}_{j-1}v_j,\\
    & \tilde{Q}'_{\mv(k)} = Q(\tilde{z}_{\mv(k)}, \tilde{r}_{k}'), \quad  \tilde{Q}_{\mv(k)} = Q(\tilde{z}_{\mv(k)}, \tilde{r}_{k})
\end{aligned}
\end{equation*}
the corresponding notions for $B$ and $C_B$.

The measure of the $k$-th frame
$Q'_{\mv(k)} \setminus Q_{\mv(k)}$ can be estimated by
\eqn{measure}
$$
    |Q'_{\mv(k)} \setminus Q_{\mv(k)}| =  (2r_k')^n - (2r_k)^n = 2^{(-k+1)n} (\alpha_{k-1}^n - \alpha_k^n)
		\approx 2^{-nk}(\alpha_{k-1}-\alpha_k)\alpha_k^{n-1},
$$
since $\alpha_k\approx\alpha_{k+1}$.
In the $k$-th generation we have $2^{nk}$ such frames.

\begin{figure}[h t p]

\begin{tikzpicture}[scale=1]
\newcommand\Square[1]{+(-#1,-#1) rectangle +(#1,#1)}
  \foreach \x in {4}
    {\draw[fill= lightgray] (\x+1,0) \Square{1}; 
    \draw[fill= white] (\x+1,0) \Square{0.4}; 
    \draw[ultra thin,<->](\x+0.6,0)--(\x+1.4,0);
    \draw[ultra thin,<->](\x+1,-0.4)--(\x+1,0.4); 
    \draw[dotted](\x+0.6,-0.4)--(\x+0,-1);
    \draw[dotted](\x+0.6,0.4)--(\x+0,1);
    \draw[dotted](\x+1.4,-0.4)--(\x+2,-1);
    \draw[dotted](\x+1.4,0.4)--(\x+2,1);}
     \draw[ultra thin,<->](4,-1.2)--(6,-1.2);
     \node[below] at (4.5, -1.2) {$2\tilde{r}'_k$};
    
    \draw[->] (2.7,0)--(3.7,0);
    \node[above] at (3.2, 0) {$g$};
    
      \foreach \x in {-4}
    {\draw[fill= lightgray] (\x+5,0) \Square{1.5}; 
    \draw[fill= white] (\x+5,0) \Square{1.2}; 
    \draw[ultra thin,<->](\x+3.8,0)--(\x+6.2,0);
            \node[above] at (0.5, 0) {$2r_k$};
    \draw[ultra thin,<->](\x+5,-1.2)--(\x+5,1.2); 
    \draw[dotted](\x+3.8,-1.2)--(\x+3.5,-1.5);
    \draw[dotted](\x+3.8,1.2)--(\x+3.5,1.5);
    \draw[dotted](\x+6.2,-1.2)--(\x+6.5,-1.5);
    \draw[dotted](\x+6.2,1.2)--(\x+6.5,1.5);}
    \draw[ultra thin,<->](-0.5,-1.7)--(2.5,-1.7);
     \node[below] at (0, -1.7) {$2r'_k$};

\end{tikzpicture}
\caption{The mapping $g$ transforms $Q_{\mv}$ onto $\tilde{Q}_{\mv}$ (the white cube) and $Q'_{\mv}\setminus Q_{\mv}$ onto $\tilde{Q}'_{\mv}\setminus \tilde{Q}_{\mv}$ (the gray frame).}
\label{pic:g}
\end{figure}
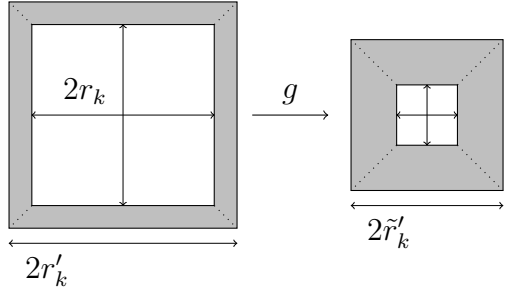

There exists a homeomorphism
$g$
which maps $C_A$ onto $C_B$ (see Figure~\ref{pic:g}).
More precisely we define this $g$ as the uniform limit of bi-Lipschitz mappings $g_k$ which map the $k$-th generation of the Cantor set $C_A$ onto the $k$-th generation of $C_B$. That is,
\eqn{star}
$$
g_0:= \id, \quad
g_{k}(x):=
g_{k-1}(x)\text{ for }x\notin \bigcup_{\mv(k) \in \mV^k} Q'_{\mv(k)}
$$
and
$$
g_{k}\text{ maps }Q_{\mv(k)}\text{ onto }\tilde{Q}_{\mv(k)}\text{ linearly for } \mv(k) \in \mV^k.
$$
Note that $g$ is locally bi-Lipschitz outside of $C_A$.
Moreover, on $Q'_{\mv(k)} \setminus Q_{\mv(k)}$ we have, analogously to \cite[proof of Theorem 4.10]{HK},
\begin{equation}\label{eq:Dg}
    |Dg(x)| \approx
    \max \left\{\frac{\tilde{r}_k}{r_k}, \frac{\tilde{r}'_k - \tilde{r}_{k}}{r'_k - r_{k}}\right\} =
    \max \left\{\frac{\beta_k}{\alpha_k}, \frac{\beta_{k-1} -\beta_{k}}{\alpha_{k-1} - \alpha_{k}}\right\}
\end{equation}
and
$$
    J_g(x) \approx
    \frac{\tilde{r}'_k - \tilde{r}_{k}}{r'_k - r_{k}}\left(\frac{\tilde{r}_k}{r_k}\right)^{n-1} =
    \frac{\frac{\tilde{r}_{k-1}}{2} - \tilde{r}_{k}}{\frac{r_{k-1}}{2} - r_{k}}\left(\frac{\tilde{r}_k}{r_k}\right)^{n-1}=\frac{\beta_{k-1}-\beta_k}{\alpha_{k-1}-\alpha_k}\left(\frac{\beta_k}{\alpha_k}\right)^{n-1}.
$$
Likewise, for $y\in \tilde{Q}'_{\mv(k)} \setminus \tilde{Q}_{\mv(k)}$ we have
\eqn{eq:Dg2}
$$
|Dg^{-1}(y)|\approx\max \left\{\frac{\alpha_k}{\beta_k},
\frac{\alpha_{k-1} - \alpha_{k}}{\beta_{k-1} -\beta_{k}}\right\}
\text{ and }
J_{g^{-1}}(y) \approx
\frac{\frac{r_{k-1}}{2} - r_{k}}{\frac{\tilde{r}_{k-1}}{2} - \tilde{r}_{k}}
		\left(\frac{r_k}{\tilde{r}_k}\right)^{n-1}.
$$

\subsection{Construction of Cantor towers and a bi-Lipschitz mapping which maps a Cantor set onto a Cantor tower}\label{map_CT}

We build a Cantor tower as in \cite{GHT}.

Suppose $n \ge 2$ and denote by $\hat{\mathbb{V}}$ the set of points
\begin{align*}
\bigl(0,0, \ldots, 0, -1+\tfrac{2j-1}{2^{n}} \bigr)
\end{align*}
where $j=1,2, \ldots, 2^{n}$. The sets
\begin{align*}
\hat{\mathbb{V}}^{k} := \hat{\mathbb{V}} \times \cdots \times \hat{\mathbb{V}}, \quad k \in \mathbb{N},
\end{align*}
serve as sets of indices in the construction of a Cantor tower. Note that the cardinality of $\hat{\mathbb{V}}$ is the same as the cardinality of $\mathbb{V}$ and therefore there is a bijection $w$ between $\mathbb{V}$ and $\hat{\mathbb{V}}$. Therefore we can also define a bijection
$$
	w^k:\mathbb{V}^k \to\hat{\mathbb{V}}^k, \quad w^k(\mv(k)) = \big(w(v_1), w(v_2), \dots, w(v_k)\big)
$$
and
$$
	w^{\infty}:\mathbb{V}^{\en}\to\hat{\mathbb{V}}^{\en}, \quad w^\infty(\mv) = \big(w(v_1), w(v_2), \dots \big).
$$

Suppose that $B$ is a decreasing sequence as before with $1 = \beta_{0}$, $\beta_k\approx \beta_{k+1}$ and moreover assume that $\beta_{k}>2^n \beta_{k+1}$. Define
\eqn{defhatrk}
$$
\hat{r}_{k} := 2^{-k}\beta_{k}\text{ and }\hat{r}'_{k} := 2^{-k}\beta_{k-1}.
$$
Set $\hat{z}_{0} = 0$. Then it follows that $Q(\hat{z}_{0}, \hat{r}_{0}) = [-1,1]^{n}$ and we proceed further by induction. For $\hat{\ve}(k) := (\hat{v}_{1}, \hat{v}_{2}, \ldots, \hat{v}_{k}) \in \hat{\mathbb{V}}^{k}$ we define (see Figure~\ref{pic:tower})
\eqn{qqq}
$$
\begin{aligned}
&\hat{z}_{\hat{\ve}(k)} := \hat{z}_{(\hat{v}_{1}, \hat{v}_{2}, \ldots, \hat{v}_{k-1})} + \hat{r}_{k-1}\hat{v}_{k} = \sum_{j=1}^{k} \hat{r}_{j-1}\hat{v}_{j},\\
&\hat{Q}_{\hat{\ve}(k)}' := Q(\hat{z}_{\hat{\ve}(k)}, \hat{r}'_{k}) \text{ and } \hat{Q}_{\hat{\ve}(k)} := Q(\hat{z}_{\hat{\ve}(k)}, \hat{r}_{k})
\end{aligned}
$$

We then define the Cantor tower $C_B^T$ as
$$
    C_B^T := \bigcap_{k=1}^{\infty} \bigcup_{\hat{\ve}(k)\in \hat{\mathbb{V}}^k} \hat{Q}_{\hat{\ve}(k)}.
$$

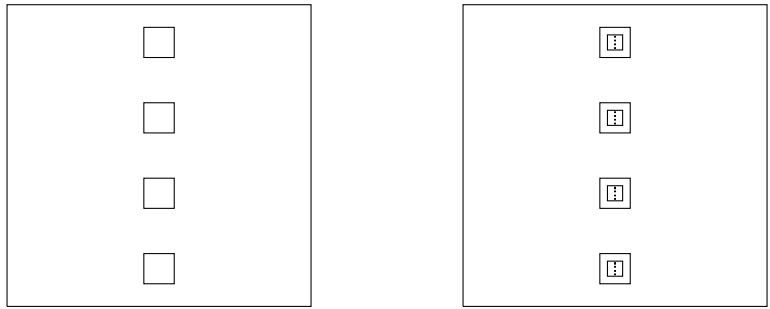
\begin{figure}[ht]
\begin{center}
\begin{tikzpicture}[scale=0.50]

\draw (-4,-4)--(4,-4)--(4,4)--(-4,4)--(-4,-4);

\draw (-0.4,-3.4) rectangle (0.4,-2.6);

\draw (-0.4,-1.4) rectangle (0.4,-0.6);

\draw (-0.4,0.6) rectangle (0.4,1.4);

\draw (-0.4,2.6) rectangle (0.4,3.4);


\draw (8,-4)--(16,-4)--(16,4)--(8,4)--(8,-4);

\draw (11.6,-3.4) rectangle (12.4,-2.6);

\draw (11.8,-3.2) rectangle (12.2,-2.8);
\draw (11.99,-3.16) rectangle (12.01,-3.14);
\draw (11.99,-3.07) rectangle (12.01,-3.05);
\draw (11.99,-2.96) rectangle (12.01,-2.94);
\draw (11.99,-2.88) rectangle (12.01,-2.84);


\draw (11.6,-1.4) rectangle (12.4,-0.6);

\draw (11.8,-1.2) rectangle (12.2,-0.8);
\draw (11.99,-1.16) rectangle (12.01,-1.14);
\draw (11.99,-1.07) rectangle (12.01,-1.05);
\draw (11.99,-0.96) rectangle (12.01,-0.94);
\draw (11.99,-0.88) rectangle (12.01,-0.84);

\draw (11.6,0.6) rectangle (12.4,1.4);

\draw (11.8,0.8) rectangle (12.2,1.2);
\draw (11.99,0.84) rectangle (12.01,0.86);
\draw (11.99,0.94) rectangle (12.01,0.96);
\draw (11.99,1.04) rectangle (12.01,1.06);
\draw (11.99,1.14) rectangle (12.01,1.16);


\draw (11.6,2.6) rectangle (12.4,3.4);

\draw (11.8,2.8) rectangle (12.2,3.2);
\draw (11.99,2.84) rectangle (12.01,2.86);
\draw (11.99,2.94) rectangle (12.01,2.96);
\draw (11.99,3.04) rectangle (12.01,3.06);
\draw (11.99,3.14) rectangle (12.01,3.16);

\end{tikzpicture}
\end{center}
\caption{Cubes $\hat{Q}_{\hat{\ve}(k)}$ and $\hat{Q}'_{\hat{\ve}(k)}$ for $k=1$, $2$ in the construction of the Cantor tower.}
\label{pic:tower}
\end{figure}

Let us now define the Cantor set $C_{B}$ as in subsection~\ref{ssec:CS} by choosing
\eqn{defbeta}
$$\beta_{k} := 2^{-k\beta},$$ 
where $\beta \geq n+1$. From \eqref{size_cantor} we know that $|C_B|=0$. Using this sequence we also fix the corresponding Cantor tower $C_{B}^{T}$.

The following theorem from \cite[Proposition 2.4]{GHT} gives us a bi-Lipschitz mapping $L \colon \rn \to \rn$ which maps the Cantor set $C_{B}$ onto the Cantor tower $C_{B}^{T}$.

\prt{Theorem}
\begin{proclaim}\label{Bilip}
Suppose that $C_{B}$ is the Cantor set and $C_{B}^{T}$ is the Cantor tower in $\rn$ defined by the sequence
\begin{align*}
\beta_{k}= 2^{-k\beta} \, ,
\end{align*}
where $\beta\geq n+1$. Then there is a bi-Lipschitz mapping $L \colon \rn \to \rn$ which takes $C_{B}$ onto $C_{B}^{T}$.
Moreover,
\eqn{goodmap}
$$
 L^{-1}(\hat{Q}_{\hat\ve(i)})=\tilde{Q}_{\ve(i)}\text{ exactly when }\hat{\ve}(i) = w^i(\ve(i)).
$$
\end{proclaim}

\subsection{Construction of tentacles and squeezing of tentacles}

Using the self-similarity \eqref{goodmap}, we define a correspondence between points $x\in C_B^T$ and $y\in C_A$ in \eqref{corespond}. Our aim in this section is to achieve the following. For each $x$ in a Cantor tower $C_B^T$, we define a curve $l_x$ and  construct an intermediary mapping which maps $l_x$ onto its corresponding point $y$ in $C_A$. 
We define this mapping as a composition of three mappings. First, we find a mapping $h$, which maps $l_x$ onto $x$ and maps $C_B^T$ onto itself.
Then, with the help of the bi-Lipschitz mapping $L^{-1}$ from Theorem~\ref{Bilip} we map $C^T_B$ onto $C_B$ and finally we map $C_B$ onto $C_A$ homeomorphically by $g^{-1}$ (the mapping from subsection~\ref{ssec:CS}),
i.e., the intermediary mapping in question is
$$
g^{-1}\circ L^{-1}\circ h \quad \text{and} \quad g^{-1}\circ L^{-1}\circ h(l_x) =\{ y\}.
$$

Further, we define the curve $q_y:=g^{-1}\big(L^{-1}(l_x)\big)$ and we define the mapping $\tilde{f}$ from Theorem~\ref{example} as $\tilde{f}= g^{-1}\circ L^{-1}\circ h\circ L\circ g$ (for an overview of the process see Figure~\ref{pic:the_mapping-1}). We show that
$$
|\bigcup_{y\in C_A}q_y\setminus C_A | =0
$$
 and $\tilde{f}$ maps $q_y$ onto $y$ for each $y\in C_A$ and therefore $\tilde{f}$ does not satisfy the Lusin $(N)$ condition (see Remark \ref{final_remark} for further discussion).

We start our exposition by defining the correspondence between points in $C_A$ and $C_B^T$. For each $x\in C_B^T$ there is exactly one $\hat{\ve}\in \hat{\mathbb{V}}^{\en}$ such that $\hat{z}_{\hat{\ve}} = x$, and for each $y\in C_A$ there is exactly one $\mv\in \mathbb{V}^{\en}$ such that $z_{\mv} = y$. Since $w^{\infty}$ is a bijection of $\mathbb{V}^{\en}$ and $\hat{\mathbb{V}}^{\en}$, we have that for each $x\in C_B^T$ there is exactly one $y\in C_A$ such that
\begin{equation}\label{corespond}
	x = L(g(y)) \text{ and this happens exactly when }w^{\infty}(\mv) = \hat{\ve}. 
\end{equation}

Now our aim is to define $l_x$ for each $x\in C_B^T$. We achieve this as follows; for each $\hat{Q}_{\hat{\ve}(k)}$ we define a tentacle $T_{\hat{\ve}(k)}$ (a long and thin polyhedron containing $\hat{Q}_{\hat{\ve}(k)}$) and we set
\begin{equation}\label{def:l}
    l_x:=\bigcap_{k=1}^{\infty}T_{\hat{\ve}(k)}.
\end{equation}
We define our tentacles by firstly defining tentative ``straight'' tentacles $T^S_{\hat{\ve}(k)}$ and then adjusting them so that
$$
T_{\hat{\ve}(k+1)}\subseteq T_{\hat{\ve}(k)}\text{ whenever }\hat{\ve}(k+1)\text{ is a continuation of }\hat{\ve}(k),
$$
(by which we mean that the first $k$ terms of $\hat{\ve}(k+1)$ are exactly $\hat{\ve}(k)$), see Figure~\ref{pic:two_gen}.

Recall the parameter $\beta$ from \eqref{defbeta}
and, from \eqref{defhatrk}, that $\hat{r}_k=2^{-k}\beta_k=2^{-k(\beta+1)}$. For $k\in \mathbb{N}$ we define
$$
    a_k  := 1 -\sum_{i=0}^{k} \hat{r}_{i+2} \approx 1,\quad
    c_k  := 1 -\sum_{i=0}^{k-1} \hat{r}_{i+2} \approx 1,
$$
and further we fix a pair of decreasing positive sequences  $b_k$ and $d_k$ such that $b_k,d_k \in(0,\hat{r}_k)$,
\eqn{choicebd}
$$
d_{k+1}< 4^n b_k \text{ and } b_k<d_k<a_k<c_k.
$$
We determine the exact values of $b_k$ and $d_k$ at a later point in the proof using Theorem \ref{thmH} below.

For $r>0$ and $\rho_1<\rho_2$ we define a parallelepiped
$$
P(r,\varrho_1, \varrho_2): = [\varrho_1,\varrho_2)\times(-r,r)\times \cdots \times  (-r,r).
$$
For each $k$ we also define
$$
P'^S_k:=P(d_k,\hat{r}_k,c_k)
\quad \text{and} \quad
P^S_k:=P(b_k,\hat{r}_k,a_k).
$$

Now we define ``straight'' tentacles as
$$
\begin{aligned}
T'^S_k:=Q(0,\hat{r}_k)\cup P'^S_k
& \quad \text{and} \quad
T^S_k:=Q(0,\hat{r}_k)\cup P^S_k, \\
T'^S_{\hat\ve(k)}:=\hat{z}_{\hat\ve(k)}+T'^S_k
& \quad \text{and} \quad
T^S_{\hat\ve(k)}:=\hat{z}_{\hat\ve(k)}+T^S_k.
\end{aligned}
$$
Both $T'^S_k$ and $T^S_k$ clearly contain $Q(0,\hat{r}_k)$ and note that $T^S_k\subseteq T'^S_k$ as $c_k>a_k$ and $d_k>b_k$.
Moreover, both $P^S_k$ and $P'^S_k$ have a common side with $Q(0,\hat{r}_k)$ and thus $T'^S_k$ is connected. Throughout the following, when we say ``length of a tentacle'' we mean the diameter of its projection onto $\er\times\{0\}^{n-1}$. The length of each tentacle $T^S_k$ is at least $a_k>1 - \frac{1}{1 -  2^{-\beta-1}}$ and hence
\begin{equation}\label{def:ls}
l^S:=\bigcap_{k=1}^{\infty}T^S_{k}\text{ is a segment of positive length}.
\end{equation}

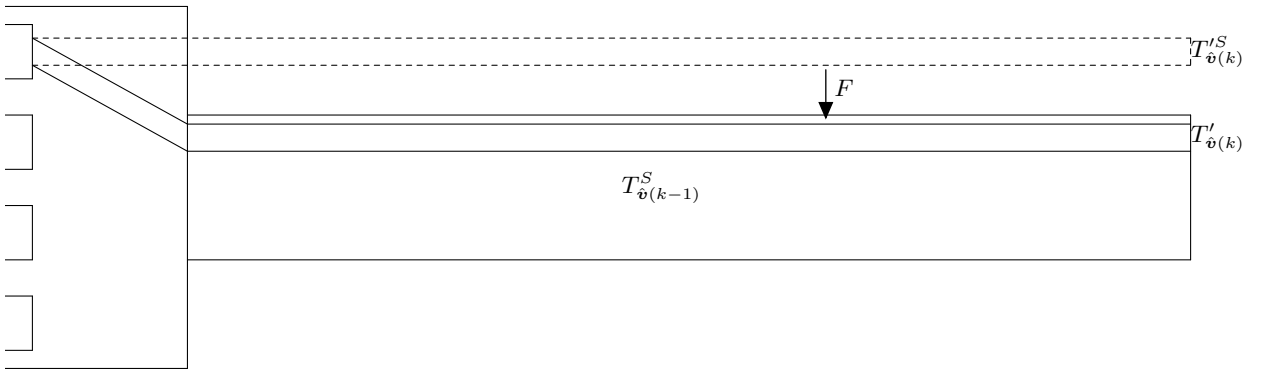
\begin{figure}[h]
	\begin{center}
		\begin{tikzpicture}[line cap=round,line join=round,>=triangle 45,x=0.15cm,y=0.15cm,scale=0.8]
			\clip(0,-22) rectangle (140,22);
			\draw (20,-20)-- (20,20);
			\draw (20,20)-- (-20,20);
			\draw (-20,-20)-- (20,-20);
			\draw (3,12)-- (3,18);
			\draw (3,18)-- (-3,18);
			\draw (-3,12)-- (3,12);
			\draw (3,2)-- (3,8);
			\draw (3,8)-- (-3,8);
			\draw (-3,2)-- (3,2);
			\draw (3,-8)-- (3,-2);
			\draw (3,-2)-- (-3,-2);
			\draw (-3,-8)-- (3,-8);
			\draw (3,-18)-- (3,-12);
			\draw (3,-12)-- (-3,-12);
			\draw (-3,-18)-- (3,-18);
			
			\draw (20,8)-- (130,8);
			\draw (130,8)-- (130,-8);
			\draw (130,-8)-- (20,-8);
			\draw (3,16.5)-- (20,7);
			\draw (3,13.5)-- (20,4);
			\draw (20,7)-- (130,7);
			\draw (20,4)-- (130,4);
			\draw [dash pattern=on 2pt off 2pt] (3,16.5)-- (130,16.5);
			\draw [dash pattern=on 2pt off 2pt] (3,13.5)-- (130,13.5);
			\draw [dash pattern=on 2pt off 2pt] (130,16.5)-- (130,13.5);
			\draw [->,line width=0.5pt] (90,13) -- (90,7.5);
			\draw[color=black] (72,0) node {\scriptsize $T^S_{\hat{\ve}(k-1)}$};
			\draw[color=black] (133,5.5) node {\scriptsize $T'_{\hat{\ve}(k)}$};
			\draw[color=black] (133,15) node {\scriptsize $T'^S_{\hat{\ve}(k)}$};
			\draw[color=black] (92,11) node {\scriptsize $F$};
		\end{tikzpicture}
	\end{center}
	\caption{Two generations of tentacles.}
	\label{pic:two_gen}
\end{figure}

The ``real'' tentacles $T_{\hat{\ve}(k)}$ are formed from the tentative ``straight'' tentacles $T^S_{\hat{\ve}(k)}$ as depicted in Fig. \ref{pic:two_gen}. We want that
\eqn{subset}
$$
T'_{\hat{\ve}(k+1)}\subseteq T_{\hat{\ve}(k)}\text{ whenever }\hat{\ve}(k+1)\text{ is a continuation of }\hat{\ve}(k).
$$
A detailed description of the construction can be found in \cite[subsection 3.2]{BHM}. We only need to know that there is a bi-Lipschitz mapping $F$ between $T_{\hat{\ve}(k)}$ and $T^S_{\hat{\ve}(k)}$ which does not change the volume nor decrease the length of tentacles. In the following, we use the notation 
\eqn{znaceni}
$$
\begin{aligned}
T_{\hat{\ve}(k)}&=Q({\hat{z}}_{\hat{\ve}(k)},\hat{r}_k)\cup P_{\hat{\ve}(k)}=Q({\hat{z}}_{\hat{\ve}(k)},\hat{r}_k) \cup F(P^S_k + z_{\hat{\ve}(k)})\\
\text{ and }
T'_{\hat{\ve}(k)}&=Q({\hat{z}}_{\hat{\ve}(k)},\hat{r}_k)\cup P'_{\hat{\ve}(k)}=Q'({\hat{z}}_{\hat{\ve}(k)},\hat{r}_k) \cup F(P'^S_k + z_{\hat{\ve}(k)}),\\
\end{aligned}
$$ 
where $P_{\hat{\ve}(k)}$ and $P'_{\hat{\ve}(k)}$ are the corresponding $F$ images of $P_k^S$ and $P'^S_k$. 
We have that
\begin{equation}\label{meas1}
|P_{\hat{\ve}(k)}|=|P^S_k | \approx a_k \cdot (2b_k)^{n-1} \approx b_k^{n-1}\text{ and }
|P'_{\hat{\ve}(k)}|=|P'^S_k|\approx c_k \cdot (2d_k)^{n-1} \approx d_k^{n-1}.
\end{equation}

Now we define a map $h$ that ``squeezes'' the tentacles so that their length goes to zero (and in the limit their intersection $l_x$ is mapped onto $x\in C_B^T$).
Analogously as above we define parameters which describe the sizes of squeezed tentacles. We set
$$
\begin{aligned}
    \tilde{a}_k & = 2 \hat{r}_{k}\approx 2^{-k(\beta+1)},\\
    \tilde{c}_k & = \tilde{a}_{k-1} = 2 \hat{r}_{k-1} \approx 2^{-k(\beta+1)}.\\
\end{aligned}
$$
With these parameters we consider for each $k\in \en$
$$
\tilde{P}'^S_k:=P(d_k,\hat{r}_k,\tilde{c}_k)
\quad \text{and} \quad
\tilde{P}^S_k:=P(b_k,\hat{r}_k,\tilde{a}_k).
$$
Now the squeezed tentacles are defined by
$$
\begin{aligned}
\tilde{T}'^S_k:=Q(0,\hat{r}_k)\cup \tilde{P}'^S_k
& \quad \text{and} \quad
\tilde{T}^S_k:=Q(0,\hat{r}_k)\cup \tilde{P}^S_k, \\
\tilde{T}'^S_{\hat\ve(k)}:=\hat{z}_{\hat\ve(k)}+\tilde{T}'^S_k
& \quad \text{and} \quad
\tilde{T}^S_{\hat\ve(k)}:=\hat{z}_{\hat\ve(k)}+\tilde{T}^S_k.
\end{aligned}
$$
As before we adjust those tentacles so that we have an analogy of \eqref{subset} and we define the squeezed tentacles similarly to \eqref{znaceni} 
$$
\tilde{T}_{\hat{\ve}(k)}=Q({\hat{z}}_{\hat{\ve}(k)},\hat{r}_k)\cup \tilde{P}_{\hat{\ve}(k)}\text{ and }
\tilde{T}'_{\hat{\ve}(k)}=Q({\hat{z}}_{\hat{\ve}(k)},\hat{r}_k)\cup \tilde{P}'_{\hat{\ve}(k)}.
$$
Note that $\lim_{k\to\infty} \tilde{c}_k=\lim_{k\to\infty} \tilde{a}_k=0$ so the length of these tentacles is small and
$$
	\bigcap_{k=1}^{\infty}\tilde{T}'_{\hat{\ve}(k)} = \Big\{\sum_{j=1}^{\infty} \hat{r}_{j-1}\hat{v}_{j}\Big\} \text{ for each } \hat{\ve}\in \hat{\mathbb{V}}^{\en}.
$$
Therefore, in the limit we have $h(l_x) = \{x\}$ (see \eqref{def:l}).

We need the following result from \cite[Theorem 3.3]{BHM} which tells us that the map which squeezes the tentacles can be constructed in such a way that its derivative has arbitrarily small $L^{n-1}$-modulus.

\prt{Theorem}
\begin{proclaim}\label{thmH}
Let $n\geq 3$, $\delta_k>0$, $\beta\geq n+1$ and $k\in\en$. Then we can find small enough $d_k>b_k>0$ and a sequence of bi-Lipschitz mappings
 $h_k\colon Q(0,1)\to Q(0,1)$ such that $h_0(x)=x$ for every $x\in Q(0,1)$,
\begin{equation}\label{Similarity}
	h_k(x)=h_{k-1}(x)\text{ for }
x\in \rn \setminus \bigcup_{\hat{\ve}(k)\in \hat{\mathbb{V}}^{k}} P_{\hat{\ve}(k)}',
\end{equation}
\begin{equation}\label{Identitarianism}
h_k(x)=x\text{ for }x\in \hat{Q}_{\hat{\ve}(k)}
\end{equation}
and
\begin{equation}\label{Imaging}
	h_k(P_{\hat{\ve}(k)})=\tilde{P}_{\hat{\ve}(k)}.
\end{equation}
We define $M_k := \bigcup_{\hat{\ve}(k)\in\hat{\mathbb{V}}^{k}} P_{\hat{\ve}(k)}'$. We can estimate the integral of its derivative as
\eqn{defsmallhk}
$$
\int_{M_k} |Dh_k(x)|^{n-1}\; dx\leq \delta_k.
$$

Moreover, the pointwise limit $h$ of $h_k$ is continuous,
$J_{h}(x) > 0$ a.e., and
$$
h(l_x)=\{x\}\text{ for every }x\in C_B^T,
$$
where
$l_x$ is defined by \eqref{def:l}.
\end{proclaim}

Now we are ready to construct our counterexample. 

\begin{proof}[Proof of Theorem \ref{example}] 
We use the sequences
\eqn{chooseseq}
$$
\alpha_k = \frac{1}{2}\left(1+2^{-k\beta}\right)\text{ and }\beta_k = 2^{-k \beta}\text{ with }\beta\geq n+1
$$
to define Cantor type sets $C_A$, $C_B$ and $C_B^T$.
We set
\eqn{defdeltak2}
$$
\delta_k=\frac{2^{-k\beta (2n-1)}}{k^2}.
$$
Given this $\delta_k$ we find $d_k>b_k>0$ with 
\eqn{bksmall}
$$
b_k<8^{-k}
$$
by Theorem~\ref{thmH} so that we have \eqref{defsmallhk}.

We define the mapping $\tilde{f}$ as the pointwise limit of
$$
\f_k(y) = g_k^{-1} \circ L^{-1} \circ h_k \circ L \circ g_k(y)
$$
almost everywhere (see Fig.~\ref{pic:the_mapping-1}).
For $y\in C_A$ we know that $L \circ g(y)=x\in C_B^T$. Further, $h_k(x)=x$ for $x\in C_B^T$ and hence it is easy to see that the pointwise limit is 
the identity on $C_A$.
Therefore, we see at once that the pointwise limit of
$\f_k$ is well-defined as
$$
\f(y)=g^{-1} \circ L^{-1} \circ h \circ L \circ g(y)\text{ at every point } y \in Q(0,1).
$$
Since $g$ and $L$ are homeomorphisms and $h$ is continuous we obtain that $\f$ is continuous on $Q(0,1)$. 

Recall  the definition of $l_x$ from \eqref{def:l}. We claim that \eqref{bksmall}, \eqref{defhatrk} and \eqref{chooseseq} imply that the set
$$
\Upsilon:=\bigcup_{x\in C_B^T} l_x \quad \text{ satisfies } \quad |\Upsilon|\leq C\lim_{k\to\infty} 2^{kn}(\hat{r}_k^n+b_k^{n-1})=0
$$
as we have $2^{kn}$ tentacles in the $k$-th step of the construction and each tentacle contains a cube of sidelength $\hat{r}_k$ and the image under $F$ of a rectangle with $n-1$ sides of size $b_k$ (see \eqref{znaceni}). 
Since $g$ is locally bi-Lipschitz outside of $C_A$ and maps $C_A$ onto $C_B$, and since $L$ is bi-Lipschitz and maps $C_B$ onto $C_B^T$, we obtain that 
\begin{equation}\label{BOOM}
	g^{-1}(L^{-1}(\Upsilon))=C_A\cup \tilde{\Upsilon}\text{ where }|\tilde{\Upsilon}|=0. 
\end{equation}
Now $L(g(\tilde{\Upsilon}))=\Upsilon\setminus C_B^T$ and $h(l_x)=\{x\}$ so our $\tilde{f}=g^{-1} \circ L^{-1} \circ h \circ L \circ g$ maps $\tilde{\Upsilon}$ onto $C_A$. It follows that a set of zero measure $\tilde{\Upsilon}$ is mapped onto a set of positive measure $C_A$ and hence our $\tilde{f}$ fails the Lusin $(N)$ condition.

It is not difficult to see that $\f$ is locally bi-Lipschitz on
$[-1,1]^n\setminus (C_A\cup \tilde{\Upsilon})$ and hence we can use the composition formula for derivatives to obtain
$$
J_{\tilde{f}}(y)=J_{g^{-1}}J_{L^{-1}} J_{h}J_L J_g >0\text{ for a.e. }x\in[-1,1]^n\setminus C_A.
$$
For $y\in C_A$ we know that $f(y)=y$ and hence $J_{f}=1$ for a.e. $x\in C_A$ once we show that $f\in W^{1,1}$ since the weak derivative is equal to the approximative derivative a.e.

\begin{figure}[h]
\begin{center}
\begin{tikzpicture}[line cap=round,line join=round,>=triangle 45,x=0.36cm,y=0.36cm]
\clip(-26,-4) rectangle (8,18);

\draw (-25,17)-- (-25,9)-- (-17,9)-- (-17,17) -- cycle;
\draw (-24.5,13.5) -- (-24.5,16.5) -- (-21.5,16.5) -- (-21.5,13.5) -- cycle;
\draw (-21.5,15.5) -- (-20.7,13.3) -- (-20.7,13.1) -- (-21.5,14.5) -- cycle;
\fill[fill=black,fill opacity=0.1] (-21.5,15.5) -- (-20.7,13.3) -- (-20.7,13.1) -- (-21.5,14.5) -- cycle;
\draw (-20.5,13.5) -- (-20.5,16.5) -- (-17.5,16.5) -- (-17.5,13.5) -- cycle;
\draw (-17.5,14.5) -- (-17.2,15) --  (-17.4,15.8) -- (-17.5,15.5) -- cycle;
\fill[fill=black,fill opacity=0.1] (-17.5,14.5) -- (-17.2,15) -- (-17.4,15.8) -- (-17.5,15.5) -- cycle;
\draw (-20.5,9.5) -- (-20.5,12.5) -- (-17.5,12.5) -- (-17.5,9.5) -- cycle;
\draw (-17.5,11.5) -- (-17.2,11)  -- (-17.4,10.3) -- (-17.5,10.5) -- cycle;
\fill[fill=black,fill opacity=0.1] (-17.5,11.5) -- (-17.2,11)  -- (-17.4,10.3) -- (-17.5,10.5) -- cycle;
\draw (-24.5,9.5) -- (-24.5,12.5) -- (-21.5,12.5) -- (-21.5,9.5) -- cycle;
\draw (-21.5,10.5) -- (-20.7,12.7) -- (-20.7,12.9) -- (-21.5,11.5) -- cycle;
\fill[fill=black,fill opacity=0.1] (-21.5,10.5) -- (-20.7,12.7) -- (-20.7,12.9) -- (-21.5,11.5) -- cycle;

\draw (-13,17)-- (-13,9)-- (-5,9)-- (-5,17) -- cycle;
\draw (-11.3,15.3) -- (-11.3,14.7) -- (-10.7,14.7) -- (-10.7,15.3) -- cycle;
\draw (-10.7,14.9) -- (-7.8,13.9) -- (-7.8,14.1) -- (-10.7,15.1) -- cycle;
\fill[fill=black,fill opacity=0.1] (-10.7,14.9) -- (-7.8,13.9) -- (-7.8,14.1) -- (-10.7,15.1) -- cycle;
\draw (-6.7,15.3) -- (-6.7,14.7) -- (-7.3,14.7) -- (-7.3,15.3) -- cycle;
\draw (-6.7,15.1) -- (-6.6,15.3) -- (-7,16.1) --  (-7,16.3) -- (-6.4,15.3) -- (-6.4,15.2) -- (-6.7,14.9) -- cycle;
\fill[fill=black,fill opacity=0.1] (-6.7,15.1) -- (-6.6,15.3) -- (-7,16.1) --  (-7,16.3) -- (-6.4,15.3) -- (-6.4,15.2) -- (-6.7,14.9) -- cycle;
\draw (-6.7,10.7) -- (-6.7,11.3) -- (-7.3,11.3) -- (-7.3,10.7) -- cycle;
\draw (-6.7,10.9) -- (-6.6,10.7) -- (-7,9.9) -- (-7,9.7) -- (-6.4,10.7) -- (-6.4,10.8) -- (-6.7,11.1) -- cycle;
\fill[fill=black,fill opacity=0.1] (-6.7,10.9) -- (-6.6,10.7) -- (-7,9.9) -- (-7,9.7) -- (-6.4,10.7) -- (-6.4,10.8) -- (-6.7,11.1) -- cycle;
\draw (-11.3,10.7) -- (-11.3,11.3) -- (-10.7,11.3) -- (-10.7,10.7) -- cycle;
\draw (-10.7,11.1) -- (-7.8,12.1) -- (-7.8,11.9) -- (-10.7,10.9) -- cycle;
\fill[fill=black,fill opacity=0.1] (-10.7,11.1) -- (-7.8,12.1) -- (-7.8,11.9) -- (-10.7,10.9) -- cycle;

\draw (-1,17)-- (-1,9)-- (7,9)-- (7,17) -- cycle;
\draw (3.3,15.7) -- (2.7,15.7) -- (2.7,16.3) -- (3.3,16.3) -- cycle;
\draw (3.3,16.1) -- (3.7,16.1) -- (3.7,15.9) -- (3.3,15.9) -- cycle;
\fill[fill=black,fill opacity=0.1] (3.3,16.1) -- (3.7,16.1) -- (3.7,15.9) -- (3.3,15.9) -- cycle;
\draw (2.7,14.3) -- (3.3,14.3) -- (3.3,13.7) -- (2.7,13.7) -- cycle;
\draw (3.3,13.9) -- (3.7,13.9) -- (3.7,14.1) -- (3.3,14.1) -- cycle;
\fill[fill=black,fill opacity=0.1] (3.3,13.9) -- (3.7,13.9) -- (3.7,14.1) -- (3.3,14.1) -- cycle;
\draw (3.3,11.7) -- (2.7,11.7) -- (2.7,12.3) -- (3.3,12.3) -- cycle;
\draw (3.3,12.1) -- (3.7,12.1) -- (3.7,11.9) -- (3.3,11.9) -- cycle;
\fill[fill=black,fill opacity=0.1] (3.3,12.1) -- (3.7,12.1) -- (3.7,11.9) -- (3.3,11.9) -- cycle;
\draw (2.7,10.3) -- (3.3,10.3) -- (3.3,9.7) -- (2.7,9.7) -- cycle;
\draw (3.3,9.9) -- (3.7,9.9) -- (3.7,10.1) -- (3.3,10.1) -- cycle;
\fill[fill=black,fill opacity=0.1] (3.3,9.9) -- (3.7,9.9) -- (3.7,10.1) -- (3.3,10.1) -- cycle;

\draw (-1,-3) -- (-1,5) -- (7,5) -- (7,-3) -- cycle;
\draw (2.7,3.7)-- (3.3,3.7)-- (3.3,4.3) -- (2.7,4.3)-- cycle;
\draw (3.3,3.9)-- (6.3,3.9)-- (6.3,4.1)-- (3.3,4.1) -- cycle;
\fill[fill=black,fill opacity=0.1] (3.3,3.9)-- (6.3,3.9)-- (6.3,4.1)-- (3.3,4.1) -- cycle;
\draw (3.3,2.3) -- (2.7,2.3) -- (2.7,1.7) -- (3.3,1.7) -- cycle;
\draw (3.3,2.1) -- (6.3,2.1) -- (6.3,1.9) -- (3.3,1.9) -- cycle;
\fill[fill=black,fill opacity=0.1] (3.3,2.1) -- (6.3,2.1) -- (6.3,1.9) -- (3.3,1.9) -- cycle;
\draw (2.7,-0.3) -- (3.3,-0.3) -- (3.3,0.3) -- (2.7,0.3) -- cycle;
\draw (3.3,-0.1) -- (6.3,-0.1) -- (6.3,0.1) -- (3.3,0.1) -- cycle;
\fill[fill=black,fill opacity=0.1] (3.3,-0.1) -- (6.3,-0.1) -- (6.3,0.1) -- (3.3,0.1) -- cycle;
\draw (3.3,-1.7) -- (2.7,-1.7) -- (2.7,-2.3) -- (3.3,-2.3) -- cycle;
\draw (3.3,-1.9) -- (6.3,-1.9) -- (6.3,-2.1) -- (3.3,-2.1) -- cycle;
\fill[fill=black,fill opacity=0.1] (3.3,-1.9) -- (6.3,-1.9) -- (6.3,-2.1) -- (3.3,-2.1) -- cycle;

\draw (-13,-3) -- (-13,5) -- (-5,5) -- (-5,-3) -- cycle;
\draw (-11.3,3.3) -- (-11.3,2.7) -- (-10.7,2.7) -- (-10.7,3.3) -- cycle;
\draw (-10.7,3.1) -- (-8.8,2.6) -- (-5.4,2.1) -- (-5.4,1.9) -- (-9.2,2.5) -- (-10.7,2.9) -- cycle;
\fill[fill=black,fill opacity=0.1] (-10.7,3.1) -- (-8.8,2.6) -- (-5.4,2.1) -- (-5.4,1.9) -- (-9.2,2.5) -- (-10.7,2.9) -- cycle;
\draw (-6.7,3.3) -- (-6.7,2.7) -- (-7.3,2.7) -- (-7.3,3.3) -- cycle;
\draw (-6.7,3.1) -- (-5.4,4.1) -- (-5.4,3.9) -- (-6.7,2.9) -- cycle;
\fill[fill=black,fill opacity=0.1] (-6.7,3.1) -- (-5.4,4.1) -- (-5.4,3.9) -- (-6.7,2.9) -- cycle;
\draw (-6.7,-1.3) -- (-6.7,-0.7) -- (-7.3,-0.7) -- (-7.3,-1.3) -- cycle;
\draw (-6.7,-1.1) -- (-5.4,-2.1) -- (-5.4,-1.9) -- (-6.7,-0.9) -- cycle;
\fill[fill=black,fill opacity=0.1] (-6.7,-1.1) -- (-5.4,-2.1) -- (-5.4,-1.9) -- (-6.7,-0.9) -- cycle;
\draw (-11.3,-1.3) -- (-11.3,-0.7) -- (-10.7,-0.7) -- (-10.7,-1.3) -- cycle;
\draw (-10.7,-1.1) -- (-8.8,-0.6) -- (-5.4,-0.1) -- (-5.4,0.1) -- (-9.2,-0.5) -- (-10.7,-0.9) -- cycle;
\fill[fill=black,fill opacity=0.1] (-10.7,-1.1) -- (-8.8,-0.6) -- (-5.4,-0.1) -- (-5.4,0.1) -- (-9.2,-0.5) -- (-10.7,-0.9) -- cycle;

\draw (-25,-3) -- (-25,5) -- (-17,5) -- (-17,-3) -- cycle;
\draw (-24.5,4.5)-- (-24.5,1.5)-- (-21.5,1.5)-- (-21.5,4.5) -- cycle;
\draw (-21.5,3.5) -- (-21,2.5) -- (-20.5,1.3) -- (-17.5,1.3) -- (-17.1,2.1) -- (-17.1,1.9) -- (-17.4,1.2) -- (-20.6,1.1) -- (-21.5,2.5) -- cycle;
\fill[fill=black,fill opacity=0.1] (-21.5,3.5) -- (-21,2.5) -- (-20.5,1.3) -- (-17.5,1.3) -- (-17.1,2.1) -- (-17.1,1.9) -- (-17.4,1.2) -- (-20.6,1.1) -- (-21.5,2.5) -- cycle;
\draw (-20.5,4.5)-- (-20.5,1.5)-- (-17.5,1.5)-- (-17.5,4.5) -- cycle;
\draw (-17.5,3.5) -- (-17.1,4.6) -- (-17.1,4.4) -- (-17.5,2.5) -- cycle;
\fill[fill=black,fill opacity=0.1] (-17.5,3.5) -- (-17.1,4.6) -- (-17.1,4.4) -- (-17.5,2.5) -- cycle;
\draw (-20.5,0.5)-- (-20.5,-2.5)-- (-17.5,-2.5)-- (-17.5,0.5) -- cycle;
\draw (-17.5,-1.5) -- (-17.1,-2.6) -- (-17.1,-2.4) -- (-17.5,-0.5) -- cycle;
\fill[fill=black,fill opacity=0.1] (-17.5,-1.5) -- (-17.1,-2.6) -- (-17.1,-2.4) -- (-17.5,-0.5) -- cycle;
\draw (-24.5,0.5)-- (-24.5,-2.5)-- (-21.5,-2.5)-- (-21.5,0.5) -- cycle;
\draw (-21.5,-1.5) -- (-21,-0.5) -- (-20.5,0.7) -- (-17.5,0.7) -- (-17.1,-0.1) -- (-17.1,0.1) -- (-17.4,0.8) -- (-20.6,0.9) -- (-21.5,-0.5) -- cycle;
\fill[fill=black,fill opacity=0.1] (-21.5,-1.5) -- (-21,-0.5) -- (-20.5,0.7) -- (-17.5,0.7) -- (-17.1,-0.1) -- (-17.1,0.1) -- (-17.4,0.8) -- (-20.6,0.9) -- (-21.5,-0.5) -- cycle;

\draw [->] (-13.7,13) -- (-16.3,13);
\draw [->] (-1.7,13) -- (-4.3,13);
\draw [->] (-4.3,1) -- (-1.7,1);
\draw [->] (3,5.7) -- (3,8.5);
\draw [->] (-16.3,1) -- (-13.7,1);
\begin{scriptsize}
\draw[color=black] (-15,13.7) node {$g^{-1}_k$};
\draw[color=black] (-3,13.7) node {$L^{-1}$};
\draw[color=black] (3.8,7.2) node {$h_k$};
\draw[color=black] (-2.7,1.7) node {$L$};
\draw[color=black] (-14.7,1.7) node {$g_k$};
\end{scriptsize}
\end{tikzpicture}
\end{center}
\caption{Mapping $\f_k$ which converges pointwise to $\f$.}
\label{pic:the_mapping-1}
\end{figure}
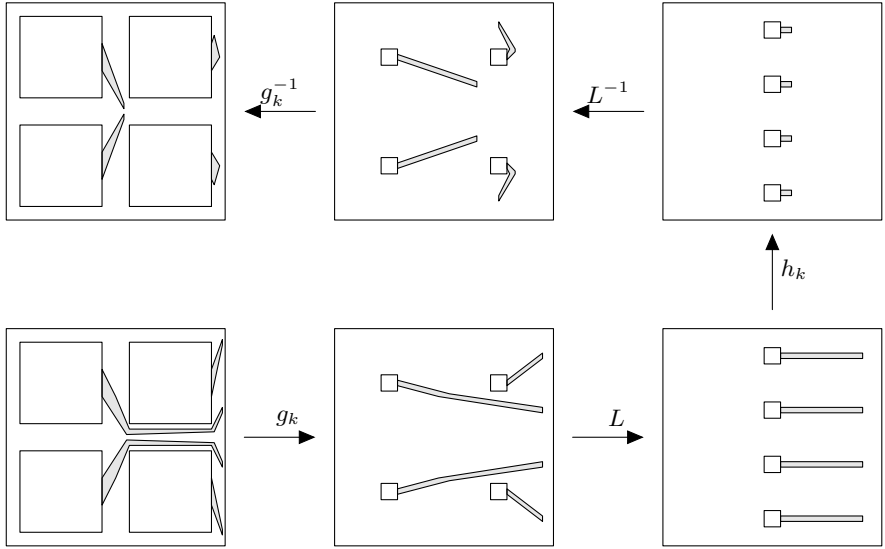

To show that $\tilde{f}$ is Sobolev (and even in $W^{1,n-1}$) we show that $D\tilde{f}_k - D\tilde{f}_{k-1}$ is Cauchy in $L^{n-1}$. Having established this fact, the Poincare inequality for mappings with zero boundary values immediately gives that $\tilde{f}_k - \tilde{f}_{k-1}$ is also Cauchy in $L^{n-1}$. Let us recall that $M_k = \bigcup_{\hat{\ve}(k)\in\hat{\mathbb{V}}^{k}} P_{\hat{\ve}(k)}'$.
By Theorem \ref{thmH} we know that
$$
	h_k(x)=h_{k-1}(x) \text{ for all } x \in \rn \setminus M_k
$$
and, by \eqref{star},
$$
	g_k(y)=g_{k-1}(y) \text{ for all } y\in \rn \setminus \bigcup_{\ve(k-1)\in\mathbb{V}^{k-1}} Q_{\ve(k-1)}.
$$
In view of  \eqref{goodmap} it follows that
$$
\f(y) = \f_k(y)=\f_{k-1}(y)\text{ for all $y$ such that both}\begin{cases}
	L(g_k(y))\in \rn \setminus  M_k \ &\text{and }\\
	y\in \rn\setminus \bigcup_{{\ve}(k-1)\in{\mathbb{V}}^{k-1}} Q_{{\ve}(k-1)}.&
\end{cases} 
$$
Note that, by \eqref{Identitarianism}, for all $x\in \bigcup_{\hat{\ve}(k-1)\in\hat{\mV}^{k-1}} \hat{Q}_{\hat{\ve}(k-1)}$ we have 
$$
	h_{k-1}(x) = x
$$
Further, by \eqref{Similarity} and the above, for $x\in \bigcup_{\hat{\ve}(k-1)\in\hat{\mV}^{k-1}} \hat{Q}_{\hat{\ve}(k-1)}\setminus M_k$ we have
$$
h_{k}(x)=h_{k-1}(x)=x.
$$
In view of $g_k(Q_{\ve(k-1)})=g_{k-1}(Q_{\ve(k-1)})=\tilde{Q}_{\ve(k-1)}$  and
$L(\tilde{Q}_{\ve(k-1)})=\hat{Q}_{\hat{\ve}(k-1)}$
we obtain by Theorem \ref{thmH}  for $y\in Q_{\ve(k-1)}\setminus g_k^{-1}(L^{-1}(M_k))$ that 
$$
\f_{k}(y)= \f_{k-1}(y)=g_{k-1}^{-1} \circ L^{-1} \circ h \circ L \circ g_{k-1}(y)=y.
$$
Therefore
\eqn{prvni2}
$$
\begin{aligned}
\int_{Q(0,1)}|D\f_{k}-D\f_{k-1}|^{n-1}&=\int_{g_k^{-1}(L^{-1}(M_k))}|D\f_{k}-D\f_{k-1}|^{n-1}\\
&\leq C \int_{g_k^{-1}(L^{-1}(M_k))}|D\f_{k}|^{n-1}+
C\int_{g_k^{-1}(L^{-1}(M_{k}))}|D\f_{k-1}|^{n-1}. \\
\end{aligned}
$$

Note that $\f_k$ is locally bi-Lipschitz on $[-1,1]^n\setminus (C_A\cup \tilde{\Upsilon})$ (as a composition of locally bi-Lipschitz mappings) and hence we can compute its derivative a.e.\ by the standard chain rule. Note that $|Dg_k|\leq C$ by \eqref{eq:Dg} and \eqref{chooseseq}. With the help of $|DL|\leq C$ and $|DL^{-1}|\leq C$ we get
$$
\begin{aligned}
|D\f_k(y)|&\leq |Dg_k^{-1}|\cdot |DL^{-1}|\cdot |D h_k|\cdot |DL|\cdot |Dg_k|\\
&\leq C \bigl|Dg^{-1}_k(L^{-1}\circ h_k\circ L\circ g_k(y))\bigr|\cdot \bigl|D h_k(L\circ g_k(y)))\bigr|.\\
\end{aligned}
$$
By the change of variables and $J_{L}\approx C$ we have
\eqn{ahoj2}
$$
\begin{aligned}
\int_{g_k^{-1}(L^{-1}(M_k))} &|D \f_k(y)|^{n-1}\; dy \leq
C\int_{g_k^{-1}(L^{-1}(M_k))} |Dg_k^{-1}|^{n-1} |D h_k|^{n-1}\frac{J_{L}}{J_{L}}\cdot \frac{J_{g_k}}{J_{g_k}} \; \\
&\leq C\int_{M_k} |Dg_k^{-1}(L^{-1}\circ h_k(x))|^{n-1} |D h_k(x)|^{n-1}\frac{1}{J_{g_k}((L\circ g_k)^{-1}(x))}\; dx.\\
\end{aligned}
$$

Note that for every $x\in P_{\hat{\ve}(k)}'\subseteq M_k$ we know that $L^{-1}\circ h_k(x)$ lies outside of $\bigcup_{{\ve(k)}\in \hat{\mathbb{V}}^k} \tilde{Q}_{{\ve}(k)}$ and hence we can use \eqref{eq:Dg2} and \eqref{chooseseq} to estimate
$$
|Dg_k^{-1}(L^{-1}\circ h_k(x))|\leq
C\max_{i=1,\hdots,k} 2^{\beta i}=C 2^{\beta k}
$$
and
$$
\frac{1}{J_{g_k}((L\circ g_k)^{-1}(x))}\leq C 2^{\beta k n}.
$$
Now \eqref{defdeltak2} and \eqref{ahoj2} imply that
$$
\int_{g_k^{-1}(L^{-1}(\tilde{M}_k))} |D \f_k(y)|^{n-1}\; dy \leq
C2^{k\beta(2n-1)}\int_{\tilde{M}_k} |D h_k(x)|^{n-1}\; dx\leq \frac{C}{k^2}.
$$
A similar estimate holds also for $D\f_{k-1}$ and hence \eqref{prvni2} implies that
$$
\int_{Q(0,1)}|D \f_{k}-D \f_{k-1}|^{n-1}\leq \frac{C}{k^2}.
$$
Since $1/k^2$ is a convergent series, $\f_k$ form a Cauchy sequence in $W^{1,n-1}$ and hence $\f\in W^{1,n-1}$.

To prove equiintegrability of $J_{\f_k}$ we use Lemma \ref{lemmaCDL}. By our construction (see Fig. \ref{pic:the_mapping-1}) we obtain that our mapping is identity outside of 
$$
g^{-1}\Bigl(L^{-1}\Bigl(\bigcup_{k=1}^{\infty}M_k \Bigr)\Bigr)
$$
so there is no problem there. Moreover, by Theorem \ref{thmH} our mapping $h_k$ maps  
$$
M_k=\bigcup_{\hat{\ve}(k)\in\hat{\mathbb{V}}^{k}} P_{\hat{\ve}(k)}'\text{ onto }\bigcup_{\hat{\ve}(k)\in\hat{\mathbb{V}}^{k}} \tilde{P}_{\hat{\ve}(k)}',
$$
moreover by an induction of \eqref{Imaging} this holds also for the limit mapping $h$. Therefore, $\f$ maps $g^{-1}\Bigl(L^{-1}\Bigl(\bigcup_{k=1}^{\infty}M_k \Bigr)\Bigr)$ onto itself.
By \eqref{BOOM}, the set 
$$
\bigcap_{k_0=1}^{\infty}g^{-1}\Bigl(L^{-1}\Bigl(\bigcup_{k=k_0}^{\infty}M_k\Bigr)\Bigr)
$$ 
has zero measure. Therefore, given $\epsilon>0$, we find $k_0$ such that for
$$
E_{k_0}:=g^{-1}\Bigl(L^{-1}\Bigl(\bigcup_{k=k_0}^{\infty}M_k\Bigr)\Bigr)\text{ we have }|\tilde{f}(E_{k_0})|=|E_{k_0}|<\frac{\epsilon}{2}.
$$
Our $\f$ is bi-Lipschitz outside of $E_{k_0}$ and hence we can find $\delta>0$ so that each set $A\subseteq Q(0,1)\setminus E_{k_0}$ with $|A|<\delta$ satisfies $|\f(A)|<\tfrac{\epsilon}{2}$. It is not difficult to verify that $(i)$ of Lemma \ref{lemmaCDL} follows at least for every $k\geq k_0$ as $\f=\f_k$ outside of the set $E_{k_0}$. As the first mappings $\f_1,\hdots,\f_{k_0}$ are bi-Lipschitz, it is easy to see that $(i)$ of Lemma \ref{lemmaCDL} holds for the whole sequence (with possibly smaller $\delta$). The other property $(ii)$ is verified analogously, so \eqref{odkaz} follows by Lemma \ref{lemmaCDL}.

It remains to prove that the pointwise and distributional Jacobians are equal.
We know from \cite[Theorem 3.1 b)]{DHM} that $\f$ satisfies the $(INV)$ condition as a strong limit of $W^{1,n-1}$ homeomorphisms. As there are no cavities created by $\f$ we claim that we can use \cite[Theorem 4.2]{CDL} to conclude that $\mathcal{J}_{\f}=J_{\f}$. In fact this is stated there only for $n=3$ but the proof works analogously in any higher dimension. The only thing we need to verify is that $\operatorname{Per}(\operatorname{im}_G(\f,(-1,1)^n))<\infty$ (see \cite{CDL} for notation and the definition of the geometric image). 
The only problematic place is the image of the Cantor-type set $C_A$, but it belongs to the geometric image (up to a set of measure zero) as our mapping $\f$ is identity on $C_A$ and thus it has distributional derivative at a.e. point of $C_A$ and $J_{\f}=1$ there. Alternatively, we can use the result of \cite[proof of Lemma 4.2]{DHMo} where we show that $\mathcal{J}_{\f}=J_{\f}$ - to verify its assumptions it is enough to verify the assumptions of \cite[Theorem 1.2]{DHMo} and to show that 
\eqn{hhh}
$$
\sup_k \int_{\Omega} A(\adj D\f_{k})<\infty\text{ for some convex }A\text{ with }\lim_{t\to\infty} \frac{A(t)}{t}=\infty.
$$
In the proof of Theorem \ref{thmH} in \cite{BHM} we use the fact that points in $\er^{n-1}$ have zero capacity in 
$W^{1,n-1}$ which allowed us to construct $h_k$ with
$$
\sup_k \int_{\Omega} |Dh_k|^{n-1}<\infty.
$$
Similarly we could use the fact that points  in $\er^{n-1}$ have zero capacity in 
$WL^{n-1}\log L$ and to construct $h_k$ with
$$
\sup_k \int_{\Omega} |Dh_k|^{n-1}\log(e+|Dh_k|)<\infty,
$$
which would give us similar condition also for integrability of $|D\f_k|$ and \eqref{hhh} would follow. 
\end{proof}

\prt{Remark}
\begin{proclaim}\label{final_remark}
It is quite interesting that the mapping $\f$ from Theorem \ref{example} fails the Lusin $(N)$ condition in a very specific way. There is a null set $\tilde{\Upsilon}$ which is mapped to a set of positive measure $f(\tilde{\Upsilon})$
but from Corollary \ref{WeakJackOnClosure} we  know that the set $f(\tilde{\Upsilon})$ must be also an image of another set of positive measure. 
It means that ``the new material is not created from nothing'' in our theory. The new material comes from a set of positive measure but it seems that another null set is mapped to the same position. We could in principle choose a different representative that would be fine but choosing a different representative if we have a continuous one is nonstandard.  
\end{proclaim}

\prt{Remark}
\begin{proclaim}
There is an error in \cite[Lemma 4.1. (iii)]{DHMo} and the Lusin $(N)$ condition does not need to hold for the limit. Indeed, the error there is the use of \cite[Lemma 2.14]{DHM} on line -4 of page 23 there - we only know that $f_G(B)\subseteq f_T(B)$ (geometric image is a subset of topological image) and not that $f(B)\subseteq f_T(B)$.  

However, we still have there that $|f_k(B)|\to |f_T(B)|$ which is enough for us to conclude $J_{f_k}\to J_f$ in $L^1$ so other parts of the paper are correct. 
\end{proclaim}


\begin{thebibliography}{00}

\bibitem{AFP}
\by{\name{Ambrosio}{L.}, \name{Fusco}{N.} and \name{Pallara}{D.}}
\book{Functions of bounded variation and free discontinuity problems}
\publ{Oxford Mathematical Monographs.
The Clarendon Press, Oxford University Press, New York, 2000}
\endbook

\bibitem{Ball}
\by{\name{Ball}{J.}}
\paper{Convexity conditions and existence theorems in nonlinear elasticity}
\jour{Arch. Rat. Mech. Anal.}
\vol{63}
\pages{337-403}
\yr{1977}
\endpaper


 \bibitem{Ball2}
 \by{\name{Ball}{J.}}
 \paper{Global invertibility of Sobolev functions and the
 interpenetration of matter}
 \jour{Proc. Roy. Soc. Edinburgh Sect. A}
 \vol{88\nom 3--4}
 \pages{315--328}
 \yr{1981}
 \endpaper

 \bibitem{BM}
 \by{\name{Ball}{J.} and \name{Murat}{F.}}
 \paper{$W^{1,p}$-quasiconvexity and variational problems for multiple integrals}
 \jour{J. Funct. Anal.}
 \vol{58\nom 3}
 \pages{225--253}
 \yr{1984}
 \endpaper

\bibitem{BHMC}
 \by{\name{Barchiesi}{M.}, \name{Henao}{D.}, \name{Mora-Corral}{C.}}
 \paper{Local invertibility in Sobolev spaces with applications to nematic elastomers and magnetoelasticity}
 \jour{Arch. Ration. Mech. Anal.}
 \vol{224}
 \pages{743--816}
 \yr{2017}
 \endpaper

\bibitem{BHMCR}
 \by{\name{Barchiesi}{M.}, \name{Henao}{D.}, \name{Mora-Corral}{C.} and \name{Rodiac}{R.}}
 \paper{Harmonic dipoles and the relaxation of the neo-Hookean energy in 3D elasticity}
 \jour{Arch. Ration. Mech. Anal.}
 \vol{247}
 \pages{Paper No. 70}
 \yr{2023}
 \endpaper



 \bibitem{BHMCR2}
 \by{\name{Barchiesi}{M.}, \name{Henao}{D.}, \name{Mora-Corral}{C.} and \name{Rodiac}{R.}}
 \paper{On the lack of compactness problem in the axisymmetric neo-Hookean model}
 \jour{Forum Math. Sigma }
 \vol{12} 
 \yr{2024}
 \pages{Paper No. e26}
 \endpaper

 \bibitem{BHMCR3}
 \by{\name{Barchiesi}{M.}, \name{Henao}{D.}, \name{Mora-Corral}{C.} and \name{Rodiac}{R.}}
 \paper{A relaxation approach to the minimization of the neo-Hookean energy in 3D}
 \jour{SIAM J. Math. Anal.}
 \vol{56} 
 \yr{2024}
 \pages{7830--7845}
 \endpaper


 \bibitem{BHM}
 \by{\name{Bouchala}{O.}, \name{Hencl}{S.} and \name{Molchanova}{A.}}
 \paper{Injectivity almost everywhere for weak limits of Sobolev homeomorphisms}
 \jour{J. Funct. Anal.}
 \vol{279} 
 \yr{2020}
 \pages{article 108658}
 \endpaper
 
 
  \bibitem{BHZ}
 \by{\name{Bouchala}{O.}, \name{Hencl}{S.} and \name{Zhu}{Z.}}
 \paper{Weak limits of Sobolev homeomorphisms are one to one}
 \jour{to appear in Calc. Var. and PDE, arXiv:2409.01260}
 \endprep

 \bibitem{Ca}
 \by{\name{Campbell}{D.}}
 \paper{Point-wise characterizations of limits of planar Sobolev homeomorphisms and their quasi-monotonicity}
 \jour{arXiv:2401.10639}
 \endprep

 \bibitem{CM}
 \by{\name{Celada}{P.}, \name{Dal Maso}{G.}}
 \paper{Further remarks on the lower semicontinuity of polyconvex integrals}
 \jour{Ann. Inst. H. Poincaré C Anal. Non Linéaire}
 \vol{11}
 \pages{661--691}
 \yr{1994}
 \endpaper

 \bibitem{CN}
 \by{\name{Ciarlet}{P. G.}, \name{Ne\v{c}as}{J.}}
 \paper{Injectivity and self-contact in nonlinear elasticity}
 \jour{Arch. Ration. Mech. Anal.}
 \vol{97}
 \pages{171--188}
 \yr{1987}
 \endpaper

\bibitem{CDL}
\by{\name{Conti}{S.}, \name{De Lellis}{C.}}
\paper{Some remarks on the theory of elasticity for compressible
Neohookean materials}
\jour{ Ann. Sc. Norm. Super. Pisa Cl. Sci.}
\vol{2}
\pages{521--549}
\yr{2003}
\endpaper


 \bibitem{DMS}
 \by{\name{Dal Maso}{G.} and \name{Sbordone}{C.}}
 \paper{Weak lower semicontinuity of polyconvex integrals: a borderline case}
 \jour{Math. Z.}
 \vol{218}
 \pages{603--609}
 \yr{1995}
 \endpaper

\bibitem{DPP}
\by{\name{De Philippis}{G.} and \name{Pratelli}{A.}}
\paper{The closure of planar diffeomorphisms in Sobolev spaces}
\jour{Ann. Inst. H. Poincare Anal. Non Lineaire}
\vol{37}
\pages{181--224}
\yr{2020}
\endpaper
	
 \bibitem{DHM}
\by{\name{Dole\v{z}alov\'a}{A.}, \name{Hencl}{S.}, \name{Mal\'y}{J.}}
\paper{Weak limit of homeomorphisms in $W^{1,n-1}$ and $(INV)$ condition}
\jour{Arch. Rational Mech. Anal.}
\vol{247}
\pages{Article No. 80, 54 pp}
\yr{2023}
\endpaper

 \bibitem{DHMo}
\by{\name{Dole\v{z}alov\'a}{A.}, \name{Hencl}{S.}, \name{Molchanova}{A.}}
\paper{Weak limit of homeomorphisms in $W^{1,n-1}$: invertibility and lower semicontinuity of energy}
\jour{ESAIM: Control Optim. Calc. Var.}
\vol{30}
\pages{article 37}
\yr{2024}
\endpaper

\bibitem{DHMS}
\by{\name{D'Onofrio}{L.}, \name{Hencl}{S.}, \name{Mal\'y}{J.}, \name{Schiattarella}{R.}}
\paper{Note on Lusin $(N)$ condition and the distributional determinan}
\jour{J. Math. Anal. Appl. }
\vol{439}
\pages{171--182}
\yr{2016}
\endpaper

\bibitem{E}
\by{\name{Eisen}{G.}}
\paper{A selection lemma for sequences of measurable sets, and lower semicontinuity of multiple integrals}
\jour{Manuscripta Math.}
\vol{27}
\pages{73--79}
\yr{1979}
\endpaper

\bibitem{FG}
\by{\name{Fonseca}{I.} and \name{Gangbo}{W.}}
\book{Degree Theory in Analysis and Applications}
\publ{Clarendon Press, Oxford, 1995}
\endbook

\bibitem{GHT}
\by{\name{Guo}{C.Y.}, \name{Hencl}{S.} and \name{Tengvall}{V.}}
\paper{Mappings of finite distortion: size of the branch set}
\jour{Adv. Calc. Var}
\vol{13}
\pages{325--360}
\yr{2020}
\endpaper

\bibitem{Ha}
\by{\name{Hajlasz}{P.}}
\paper{Change of variables formula under minimal assumptions}
\jour{Colloq. Math.}
\vol{64}
\pages{93--101}
\yr{1993}
\endpaper

\bibitem{HMC}
\by{\name{Henao}{D.} and \name{Mora-Corral}{C.}}
\paper{Lusin's condition and the distributional determinant for deformations with finite energy}
\jour{Adv. Calc. Var.}
\vol{5}
\pages{355--409}
\yr{2012}
\endpaper

\bibitem{HeMo11}
\by{\name{Henao}{D.} and \name{Mora-Corral}{C.}}
\paper{Fracture surfaces and the regularity of inverses for {BV} deformations}
\jour{Arch. Rational Mech. Anal.}
\vol{201}
\pages{575--629}
\yr{2011}
\endpaper


\bibitem{HK}
  \by{\name{Hencl}{S.} and \name{Koskela}{P.}}
  \book{Lectures on Mappings of finite distortion}
  \publ{Lecture Notes in Mathematics 2096, Springer, 2014, 176pp}
  \endbook


\bibitem{HO}
\by{\name{Hencl}{S.}, \name{Onninen}{J.}}
\paper{Jacobian of weak limits of Sobolev homeomorphisms}
\jour{Adv. in Calc. Var.,  https://doi.org/10.1515/acv-2016-0005}
\vol{11. \rm{no. 1}}
\pages{65--73}
\yr{2018}
\endpaper

 \bibitem{M}
 \by{\name{Mal\'y}{J.}}
 \paper{Weak lower semicontinuity of polyconvex integrals}
 \jour{Proc. Roy. Soc. Edinburgh Sect. A}
 \vol {123, \rm no. 4}
 \yr {1993}
 \pages {681--691}
 \endpaper
 


\bibitem{MCMC}
\by{\name{Mora-Corral}{C.} and \name{Mur-Callizo}{D.}}
\paper{Invertibility of Sobolev maps through approximate invertibility at the boundary and tangential polyconvexity}
\jour{ArXiv:2503.01795}
\endprep




\bibitem{IM}
\by{\name{Iwaniec}{T.} and \name{Martin}{G.}}
\book{Geometric function theory and nonlinear analysis}
\publ{Oxford Mathematical Monographs, Clarendon Press, Oxford 2001}
\endbook

\bibitem{IO}
\by{\name{Iwaniec}{T.} and \name{Onninen}{J.}}
\paper{Monotone Sobolev Mappings of Planar Domains and Surfaces}
\jour{Arch. Ration. Mech. Anal.}
\vol{219}
\pages{159--181}
\yr{2016}
\endpaper

\bibitem{IO2}
\by{\name{Iwaniec}{T.} and \name{Onninen}{J.}}
\paper{Limits of Sobolev homeomorphisms}
\jour{J. Eur. Math. Soc.}
\vol{19}
\pages{473--505}
\yr{2017}
\endpaper



\bibitem{MS}
\by{\name{M\"uller}{S.} and \name{Spector}{S. J.}}
\paper{An existence theory for nonlinear elasticity that
allows for cavitation}
\jour{Arch. Ration. Mech. Anal.}
\vol {131, \rm no. 1}
\yr {1995}
\pages {1--66}
\endpaper


\bibitem{MST}
\by{\name{M\"uller}{S.}, \name{Spector}{S. J.} and \name{Tang}{Q.}}
\paper{Invertibility and a topological property of Sobolev maps}
\jour{SIAM J. Math. Anal.}
\vol{27}
\yr{1996}
\pages{959--976}
\endpaper

\bibitem{MTY}
\by {\name{M\"uller}{S.}, \name{Tang}{Q.} and \name{Yan}{B. S.}}
\paper{On a new class of elastic deformations not allowing for
cavitation}
\jour{Analyse Nonlineaire}
\vol{11}
\yr{1994}
\pages{217--243}
\endpaper

\bibitem{Re}
\by{\name{Reshetnyak}{Y.\,G.}}
\book{Space Mappings with Bounded Distortion}
\publ{Transl. Math. Monographs 73, AMS, Providence, Rhode Island 1989}
\endbook

\bibitem{Sv}
\by{\name{\v Sver\'ak}{V.}}
\paper{Regularity properties of deformations with finite energy}
\jour{Arch. Ration. Mech. Anal.}
\vol{100\nom 2}
\pages{105--127}
\yr{1988}
\endpaper


\end{thebibliography}
\end{document}